\documentclass[11pt, reqno]{amsart}

\usepackage{latexsym}
\usepackage{amssymb}
\usepackage{mathrsfs}
\usepackage{amsmath}
\usepackage{fancybox,color}
\usepackage{enumerate}
\usepackage{enumitem}
\usepackage[latin1]{inputenc}
\usepackage{comment}
\usepackage{url}

\usepackage[colorlinks]{hyperref}

\newcommand{\kom}[1]{}
\renewcommand{\kom}[1]{{\bf [#1]}}

 \def\1{\raisebox{2pt}{\rm{$\chi$}}}

\newtheorem{theorem}{Theorem}[section]
\newtheorem{corollary}[theorem]{Corollary}
\newtheorem{lemma}[theorem]{Lemma}
\newtheorem{proposition}[theorem]{Proposition}
\newtheorem{definition}[theorem]{Definition}
\newtheorem{remark}[theorem]{Remark}

 \newcommand{\eps}{{\varepsilon}}
 \def\1{\raisebox{2pt}{\rm{$\chi$}}}
 
\newcommand{\abs}[1]{\left|#1\right|}

\def\vint_#1{\mathchoice%
          {\mathop{\kern 0.2em\vrule width 0.6em height 0.69678ex depth -0.58065ex
                  \kern -0.8em \intop}\nolimits_{\kern -0.4em#1}}%
          {\mathop{\kern 0.1em\vrule width 0.5em height 0.69678ex depth -0.60387ex
                  \kern -0.6em \intop}\nolimits_{#1}}%
          {\mathop{\kern 0.1em\vrule width 0.5em height 0.69678ex
              depth -0.60387ex
                  \kern -0.6em \intop}\nolimits_{#1}}%
          {\mathop{\kern 0.1em\vrule width 0.5em height 0.69678ex depth -0.60387ex
                  \kern -0.6em \intop}\nolimits_{#1}}}
\def\vintslides_#1{\mathchoice%
          {\mathop{\kern 0.1em\vrule width 0.5em height 0.697ex depth -0.581ex
                  \kern -0.6em \intop}\nolimits_{\kern -0.4em#1}}%
          {\mathop{\kern 0.1em\vrule width 0.3em height 0.697ex depth -0.604ex
                  \kern -0.4em \intop}\nolimits_{#1}}%
          {\mathop{\kern 0.1em\vrule width 0.3em height 0.697ex depth -0.604ex
                  \kern -0.4em \intop}\nolimits_{#1}}%
          {\mathop{\kern 0.1em\vrule width 0.3em height 0.697ex depth -0.604ex
                  \kern -0.4em \intop}\nolimits_{#1}}}

\newcommand{\aveint}[2]{\mathchoice%
          {\mathop{\kern 0.2em\vrule width 0.6em height 0.69678ex depth -0.58065ex
                  \kern -0.8em \intop}\nolimits_{\kern -0.45em#1}^{#2}}%
          {\mathop{\kern 0.1em\vrule width 0.5em height 0.69678ex depth -0.60387ex
                  \kern -0.6em \intop}\nolimits_{#1}^{#2}}%
          {\mathop{\kern 0.1em\vrule width 0.5em height 0.69678ex depth -0.60387ex
                  \kern -0.6em \intop}\nolimits_{#1}^{#2}}%
          {\mathop{\kern 0.1em\vrule width 0.5em height 0.69678ex depth -0.60387ex
                  \kern -0.6em \intop}\nolimits_{#1}^{#2}}}
\newcommand{\ud}{\, d}

\newcommand{\ol}{\overline}

\newcommand{\esssup}{\operatornamewithlimits{ess\, sup}}

\newcommand{\tr}{\operatorname{trace}}

\begin{document}

\title[A continuous time tug-of-war]{A continuous time tug-of-war game for parabolic $p(x,t)$-Laplace type equations}

\author[Joonas Heino]{Joonas Heino}
\address{Department of Mathematics and Statistics, University of Jyv\"{a}skyl\"{a}, PO Box 35, FI-40014 Jyv\"{a}skyl\"{a}, Finland}
\email{joonas.heino@jyu.fi}

\date{\today}
\keywords{normalized $p(x,t)$-Laplacian, parabolic partial differential equation, stochastic differential game, viscosity solution.} \subjclass[2010]{91A15,49L25,35K65}

\begin{abstract}
We formulate a stochastic differential game in continuous time that represents the unique viscosity solution to a terminal value problem for a parabolic partial differential equation involving the normalized $p(x,t)$-Laplace operator. Our game is formulated in a way that covers the full range $1<p(x,t)<\infty$. Furthermore, we prove the uniqueness of viscosity solutions to our equation in the whole space under suitable assumptions.
\end{abstract}

\maketitle

\section{Introduction}
In this paper, we study a two-player zero-sum stochastic differential game (SDG) that is defined in terms of an $n$-dimensional state process, and is driven by a $2n$-dimensional Brownian motion for $n\geq 2$. The players' impacts on the game enter in both a diffusion and a drift coefficient of the state process. The game is played in $\mathbb R^n$ until a fixed time $T>0$, and at that time a player pays the other player the amount given by a pay-off function $g$ at a current point. We show that the game has a value, and characterize the value function of the game as a viscosity solution $u$ to a parabolic terminal value problem 
\begin{align*}
\begin{cases}
\partial_tu(x,t)+\triangle_{p(x,t)}^Nu(x,t)+\sum_{i=1}^n\mu_i\frac{\partial u}{\partial x_i }(x,t)=ru(x,t)~~&\text{in }\mathbb R^n\times (0,T), \\
u(x,T)=g(x)~~&\text{on }\mathbb R^n
\end{cases}
\end{align*}
for $\mu\in \mathbb R^n$ and $r\geq 0$. Moreover, we show that the viscosity solution $u$ is unique under suitable assumptions. Here, the normalized $p(x,t)$-Laplacian is defined as 
\begin{align*}
&\triangle_{p(x,t)}^Nu(x,t)\\
&:=\bigg(\frac{p(x,t)-2}{|Du(x,t)|^2}\bigg)\sum_{i,j=1}^n\frac{\partial^2 u}{\partial x_i \partial x_j}(x,t)\frac{\partial u}{\partial x_i }(x,t)\frac{\partial u}{\partial x_j }(x,t)+\sum_{i=1}^n\frac{\partial^2 u}{\partial x_i^2}(x,t)
\end{align*}
for $x\in \mathbb R^n$ and $t\in (0,T)$, provided that $Du(x,t)\not=0$. The vector $Du=(\partial u/\partial x_1,\dots,\partial u/\partial x_n)^T$ is the gradient with respect to $x$, and the function $p:\mathbb R^n \times [0,T] \to \mathbb R$ is Lipschitz continuous with values on a compact set $[p_{\text{min}},p_{\text{max}}]$ for constants $1<p_{\text{min}}\leq p_{\text{max}}<\infty$.

This work is motivated by a connection between $p$-harmonic functions and a stochastic game called tug-of-war, see the seminal papers \cite{peresssw09,peress08, manfredipr12} in the elliptic case and \cite{manfredipr10c} in the parabolic case. Furthermore, Atar and Budhiraja \cite{atarb10} formulated a game in continuous time representing the unique viscosity solution to a certain elliptic inhomogeneous problem with the normalized $\infty$-Laplacian. The contribution of our work is the identification of a game in continuous time that corresponds to the parabolic normalized $p(x,t)$-Laplace operator. Moreover, our game covers the full range $1<p(x,t)<\infty$. In the game formulation, we increased the dimension of the Brownian motion that drives our state process to let $p$ also get values below two. This approach is new even for constant $p$.

In this work, main difficulties arise from the variable dependence in $p$ and from the unboundedness of the game domain. It is simpler to approximate viscosity solutions and to prove comparison principles to our equations without the variable dependence in $p$. Furthermore, we overcome the loss of translation invariance on the SDG by utilizing the H\"{o}lder continuity of solutions to Bellman-Isaacs type equations. Because the game domain is unbounded, we need to eliminate solutions growing too fast when $|x|\to \infty$. We show that under a linear growth bound a viscosity solution to our equation is unique.

\subsection{SDG formulation}
We fix a time $T>0$, and model $X(t),t\in[0,T]$ by a stochastic differential equation
\begin{align}
\begin{split}\label{intro stoc dynamics}
\begin{cases}
dX(s)&=\rho\big(G(s)\big)\ud s+\sigma\big(X(s),G(s)\big)\ud \ol{W}(s)\\
X(0)&=x,
\end{cases}
\end{split}
\end{align}
where $x\in \mathbb R^n$, and $\ol{W}$ is a $2n$-dimensional Brownian motion on a probability space $(\Omega,\mathcal F,\{\mathcal F_s\},\mathbb P)$ satisfying the standard assumptions. In our model, there are two competing players. We let  
$$
G(s)=\big(a(s),b(s),c(s),d(s)\big),
$$
where
$$
a(s),b(s) \in \mathbb S^{n-1},\, c(s),d(s)\in [0,\infty),\, s\in[0,T]
$$
are progressively measurable stochastic processes with respect to the filtration $\{\mathcal F_s\}$. Throughout the paper, $\mathbb S^{n-1}$ denotes the unit sphere of $\mathbb R^n$. The pairs $\big(a(s),c(s)\big)$ and $\big(b(s),d(s)\big)$ are called controls of the players. Roughly speaking, $a(s)$ and $b(s)$ are the directions, and $c(s)$ and $d(s)$ are the lengths taken by the players at the time $s$. Furthermore, let $\mu\in \mathbb R^n$. Then, for $s\in [0,T]$, we define the function $\rho$ in \eqref{intro stoc dynamics} by
$$
\rho\big(G(s)\big)=\mu + \big(c(s)+d(s)\big)\big(a(s)+b(s)\big).
$$
Recall that $p:\mathbb R^n \times [0,T] \to \mathbb R$ is a Lipschitz continuous function taking values on the compact set $[p_{\text{min}},p_{\text{max}}]$. We define the $n\times 2n$ matrix $\sigma$ in \eqref{intro stoc dynamics} to be
\begin{align*}
&\sigma\big(X(s),G(s)\big) \\
&= \Big[a(s)\sqrt{p\big(X(s),s\big)-1};~\hspace{4pt} P_{a(s)}^\bot;~\hspace{4pt} b(s)\sqrt{p\big(X(s),s\big)-1};~\hspace{4pt} P_{b(s)}^\bot\Big],
\end{align*}
where the $n\times (n-1)$ matrices $P_{a(s)}^\bot$ and $P_{b(s)}^\bot$ are defined such that the matrices
$$
P_{a(s)}^\bot \big(P_{a(s)}^\bot\big)^T \text{ and } P_{b(s)}^\bot \big(P_{b(s)}^\bot\big)^T
$$
are projections to the $(n-1)$-dimensional hyperspaces orthogonal to the vectors $a(s)$ and $b(s)$ at the time $s$, respectively. For more details on $\sigma$, see Section \ref{sec: pre} below.

We only allow players to use admissible controls. Roughly speaking, a player initially declares a bound $C<\infty$, and then plays as to keep $c(s)\leq C$ for all $s$, where $\big(a(s),c(s)\big)$ is the admissible control of the player. 
\begin{definition}\label{def:admissible controls}
Given a control $A:=\big(a(s),c(s)\big)$, that is, a progressively measurable process with respect to the Brownian filtration $\{\mathcal F_s\}$ with $a(s)\in \mathbb S^{n-1}$, 
$c(s)\in[0,\infty)$, and $s\in [0,T]$, we set
\begin{align}\label{admissible controls}
\Lambda(A)=\esssup_{\omega\in \Omega}\sup_{s\in[0,T]} c(s)\in[0,\infty].
\end{align}
Then, we define the set of admissible controls by 
\begin{align*}
\mathcal{AC}=\{A~\text{control}:\Lambda(A)<\infty\}.
\end{align*}
\end{definition}
Given an admissible control $A$, we say that the compact set $\mathbb S^{n-1}\times [0,\Lambda(A)]$ is an action set. A strategy is a response to the control of the opponent. 
\begin{definition}\label{def:admissible strategies}
A strategy is a function
$$
S: \mathcal{AC} \to \mathcal{AC}
$$
such that for all $t\in[0,T]$, if
$$
\mathbb P\big(A(s)=\tilde{A}(s)~\text{for a.e.}~s\in[0,t]\big)=1~\text{and }\Lambda(A)=\Lambda(\tilde{A}),
$$
then
$$
\mathbb P\big(S(A)(s)=S(\tilde{A})(s)~\text{for a.e.}~s\in[0,t]\big)=1~\text{and }\Lambda(S(A))=\Lambda(S(\tilde{A})).
$$
Given a strategy $S$, we set
\begin{align}\label{admissible strategy}
\Lambda(S):=\sup_{A\in\mathcal{AC}}\Lambda(S(A))\in[0,\infty].
\end{align}
Then, we define the set of admissible strategies by
\begin{align*}
\mathcal S=\{S~\text{strategy}: \Lambda(S)<\infty\}.
\end{align*}
\end{definition}
We define the lower and upper values of the game with the dynamics \eqref{intro stoc dynamics} by
\begin{align}
\begin{split}\label{intro lower and upper values}
U^-(x,t)&= \inf_{S \in \mathcal S} \sup_{A\in \mathcal{AC}} \mathbb E\Big[e^{-r(T-t)}g\big(X(T)\big)\Big],\\
U^+(x,t)&=\sup_{S \in \mathcal S} \inf_{A\in \mathcal{AC}} \mathbb E\Big[e^{-r(T-t)}g\big(X(T)\big)\Big]
\end{split}
\end{align}
for all $(x,t)\in \mathbb R^n\times [0,T]$, where $r\geq 0$, and $g$ is the pay-off function. The game starts at a position $x$ at a time $t$, and the expectation $\mathbb E$ is taken with respect to the measure $\mathbb P$. The game is said to have a value at $(x,t)$, if it holds $U^-(x,t)=U^+(x,t)$. 

\subsection{Statement of the main results}
Let us denote
\begin{align*}
&F\big((x,t), u(x,t),Du(x,t),D^2u(x,t)\big)\\
&:=\triangle_{p(x,t)}^Nu(x,t)+\sum_{i=1}^n\mu_i\frac{\partial u}{\partial x_i }(x,t) -ru(x,t)
\end{align*}
for all $(x,t)\in \mathbb R^n \times (0,T)$, where $D^2u$ is the matrix consisting of the second order derivatives with respect to $x$. We consider the terminal value problem 
\begin{align}
\begin{split}\label{limit terminal value problem}
\begin{cases}
\partial_tu+F\big((x,t),u,Du,D^2u\big)=0~~&\text{in }\mathbb R^n\times (0,T), \\
u(x,T)=g(x)~~&\text{on }\mathbb R^n,
\end{cases}
\end{split}
\end{align}
where $g$ is a positive, bounded and Lipschitz continuous function. A common notion of a weak solution to this equation is a viscosity solution. In this paper, we prove the following main result.
\begin{theorem}\label{main theorem of the paper}
Let $g$ be positive, bounded and Lipschitz continuous. Furthermore, let $U^-$ and $U^+$ be the lower and upper values of the stochastic differential game defined in \eqref{intro lower and upper values}, respectively. Then, the functions $ U^-$ and $U^+$ are viscosity solutions to \eqref{limit terminal value problem}.
\end{theorem}
For completeness, we show that a viscosity solution to \eqref{limit terminal value problem} is unique under suitable assumptions.
\begin{theorem}\label{thm uniq sol}
Let $g$ be positive, bounded and Lipschitz continuous. Then, a viscosity solution $u$ to the equation \eqref{limit terminal value problem} is unique, if $u$ satisfies a linear growth bound
\begin{align}\label{intro growth cond}
|u(x,t)|\leq c(1+|x|)
\end{align}
for all $(x,t)\in \mathbb R^n\times [0,T]$ and for $c<\infty$ independent of $x,t$.
\end{theorem}
Because $g$ is bounded, the functions $U^-$ and $U^+$ satisfy \eqref{intro growth cond}. Thus, Theorems \ref{main theorem of the paper} and \ref{thm uniq sol} imply the following.
\begin{corollary}
The game has a value at every $(x,t)\in \mathbb R^n\times [0,T]$.
\end{corollary}

As an application, one could study our model in the context of the portfolio option pricing. This would be based on the idea that, in addition to a random noise, the prices of the underlying assets are influenced by the two competing players. Roughly speaking, one can see the players as the issuer and the holder of the corresponding option. The issuer and the holder try, respectively, to manipulate the drifts and the volatilities of the assets to minimize and maximize, respectively, the expected discounted reward at the time $T$. The time $T$ can be interpreted as a maturity; it is the time on which the corresponding financial instrument must either be renewed or it will cease to exist. To some extent, we generalize the model developed by Nystr\"om and Parviainen in \cite{nystromp17}. Indeed, our contribution is the introduction of a local volatility $p$. The volatility of an asset may vary over the space and the time.

\subsection{An outline of the proofs  of Theorems \ref{main theorem of the paper} and \ref{thm uniq sol}}
Our approach is influenced by the papers \cite{swiech96,atarb10,nystromp17}. First, we examine games with uniformly bounded action sets, and in the end, let the uniform bound tend to the infinity. Here, the important step is to connect the value functions under uniformly bounded action sets to the terminal value problems of Bellman-Isaacs type equations
\begin{align}
\begin{split}\label{eq: intro 2}
\begin{cases}
\partial_tu-F_m^-\big((x,t),u,Du,D^2u\big)=0~~&\text{in }\mathbb R^n\times (0,T), \\
u(x,T)=g(x)~~&\text{on }\mathbb R^n,
\end{cases}
\end{split}
\end{align}
and
\begin{align}
\begin{split}\label{eq: intro 1}
\begin{cases}
\partial_tu-F_m^+\big((x,t),u,Du,D^2u\big)=0~~&\text{in }\mathbb R^n\times (0,T), \\
u(x,T)=g(x)~~&\text{on }\mathbb R^n.
\end{cases}
\end{split}
\end{align}
The exact definitions of $F_m^-$ and $F_m^+$ are given in Section \ref{sec: pre} below. Here, $m$ denotes the uniform bound on the controls. The uniqueness of viscosity solutions to \eqref{eq: intro 2} and \eqref{eq: intro 1}  follows, for example, from \cite{gigagis91,buckdahnl08}. Furthermore, the existence of viscosity solutions to the equations \eqref{eq: intro 2} and \eqref{eq: intro 1} follows by the construction of suitable barriers (Lemma \ref{existence}) and by the use of Perron's method.

In Section \ref{sec: bounded controls}, the main result is Lemma \ref{lemma: game value} in which we show that a lower value function with uniformly bounded controls equals to the unique solution $u_m$ to \eqref{eq: intro 2}. In the proof, we first regularize the solution $u_m$ by sup- and inf-convolutions, and then deduce the equality by utilizing Ito's formula and passing to limits.

In section \ref{sec: limit}, we examine the problem \eqref{limit terminal value problem}. First, we prove Theorem \ref{thm uniq sol}. To prove a comparison principle, we double the variables and apply the celebrated theorem of sums, see \cite{crandalli90a}. Because we only consider solutions satisfying a linear growth bound in the whole space, we utilize a quadratic barrier function for the space infinity. Furthermore, we use the Lipschitz continuity of $p$ to estimate the error coming from a penalty function. To continue, in Lemma \ref{lemma: limit eq} we show that
$$
F_m^-\to F
$$
as $m\to \infty$. Furthermore, in Lemma \ref{lemma: equicontinuity} we utilize the results of \cite{krylovs80,wang92} to show that the family
$$
\{u_m:m\geq1\}
$$
is equicontinuous. Finally by the reduction of test functions (Lemma \ref{test f reduction}) and the stability principle for viscosity solutions, we can utilize the Arzel\`{a}-Ascoli theorem to find a solution $u$ to \eqref{limit terminal value problem} and a subsequence $(u_{m_j})$ converging uniformly to $u$ as $j\to \infty$. To complete the proof of Theorem \ref{main theorem of the paper}, we also need the fact that the subsequence of the corresponding lower value functions converges to the lower value function for the game without the uniform bound on the controls. In addition, all the proofs in the context of the equation \eqref{eq: intro 1} are analogous. 

\subsection*{Acknowledgement}
The author would like to thank Mikko Parviainen for many discussions and insightful comments regarding this work.
\section{Preliminaries}\label{sec: pre}
Let $\ol{W}=(W^1,W^2)^T$ be a $2n$-dimensional Brownian motion such that $W^1=(W_1^1,\dots,W_n^1)$ and $W^2=(W_1^2,\dots,W_n^2)$ are $n$-dimensional Brownian motions. Let $(\Omega,\mathcal F,\{\mathcal F_s\},\mathbb P)$ denote a complete filtered probability space with right-continuous filtration supporting the process $\ol{W}$. As mentioned above, we consider the following stochastic differential equation
\begin{align}
\begin{split}\label{stoc dynamics}
\begin{cases}
dX(s)&=\rho\big(G(s)\big)\ud s+\sigma\big(X(s),G(s)\big)\ud \ol{W}(s)\\
X(0)&=x
\end{cases}
\end{split}
\end{align}
for $s\in [0,T]$, $T>0$ and $x\in \mathbb R^n$ with $G: [0,T] \to \mathcal{CS}$, $\rho:\mathcal{CS}\to \mathbb R^n$ and $\sigma: \mathbb R^n \times \mathcal{CS}\to M^{n\times2n}$. Here, we define $\mathcal{CS}:=\mathbb S^{n-1}\times \mathbb S^{n-1}\times [0,\infty)\times [0,\infty)$, where $\mathcal{CS}$ refers to control space. Furthermore, $M^{n\times2n}$ is the set of $n\times 2n$ matrices.

We are interested in the following form of the functions $G$, $\rho$ and $\sigma$. Let $A_1:=\big(a(s),c(s)\big)$ and $A_2:=\big(b(s),d(s)\big)$ be admissible controls of the players in the sense of Definition \ref{def:admissible controls}, respectively. Furthermore, let $\mu \in \mathbb R^n$. Then, for $s\in [0,T]$, we define
$$
G(s)=\big(a(s),b(s),c(s),d(s)\big),
$$
and
$$
\rho\big(G(s)\big)=\mu + \big(c(s)+d(s)\big)\big(a(s)+b(s)\big).
$$
Let $\nu\in \mathbb S^{n-1}$, and denote the orthogonal complement of $\nu$ by
$$
\nu^\bot:=\{z\in \mathbb R^n: \langle z,\nu\rangle=0\}.
$$
We set $P_{\nu}^\bot$ to be a $n \times (n-1)$ matrix such that the columns are $p_{\nu}^1,\dots,p_{\nu}^{n-1}$, where $\{p_{\nu}^1,\dots,p_{\nu}^{n-1}\}$ is  a fixed orthonormal basis of $\nu^\bot$,
$$
P_\nu^\bot=\big[p_\nu^1 ~\cdots ~ p_\nu^{n-1}\big].
$$
We can define the basis of $\nu^\bot$ such that the function $\nu \mapsto P_\nu^\bot $ is continuous. In addition, let $p: \mathbb R^n \times [0,T]\to \mathbb R$  be a Lipschitz continuous function  such that
\begin{equation}\label{p range}
p_{\text{min}}=\inf_{y \in\mathbb R^n\times [0,T]}p(y)>1~\text{and  } p_{\text{max}}=\sup_{y \in\mathbb R^n\times [0,T]}p(y)<\infty.
\end{equation}
With respect to the time variable $t$, we only need that $p$ is H\"{o}lder continuous for all fixed $x$, but we minimize additional technical difficulties. Now, we define the $n\times 2n$ matrix $\sigma$ to be
\begin{align*}
&\sigma\big(X(s),G(s)\big) \\
&= \Big[a(s)\sqrt{p\big(X(s),s\big)-1};~\hspace{4pt} P_{a(s)}^\bot;~\hspace{4pt} b(s)\sqrt{p\big(X(s),s\big)-1};~\hspace{4pt} P_{b(s)}^\bot\Big].
\end{align*}
By the game dynamics \eqref{stoc dynamics}, we get
\begin{align}
\begin{split}\label{comp stoc dynamics}
dX_i(s)=\,&\Big[\mu_i+\big(c(s)+d(s)\big)\big(a_i(s)+b_i(s)\big)\Big]\ud s \\
&+\sqrt{p\big(X(s),s\big)-1}\Big(a_i(s)\ud W_1^1(s)+b_i(s)\ud W^2_1(s)\Big) \\
&+\sum_{k=2}^n \big({\overset{\to}{p}}_{a(s)}^i\big)_{k-1}\ud W_k^1(s)+\sum_{k=2}^n\big({\overset{\to}{p}}_{b(s)}^i\big)_{k-1}\ud W_k^2(s)
\end{split}
\end{align}
for all $i \in \{1,\dots,n\}$. Here, $\big({\overset{\to}{p}}_{\nu}^i\big)$ denotes the $i$-th row vector of $P_{\nu}^\bot$.

By a strong solution to the stochastic differential equation \eqref{stoc dynamics}, we mean a progressively measurable process $(X(t))$ with respect to the Brownian filtration $\{\mathcal F_t\}$ such that the stochastic integral in right-hand side of \eqref{stoc dynamics} is defined and furthermore, $X(t)$ coincides with the right-hand side of \eqref{stoc dynamics} for all $t\in [0,T]$ almost surely. In addition, a strong solution is pathwise unique, if any two given solutions $\big(X(t),Y(t)\big)$ satisfy $$\mathbb P\big(\sup_{t\in[0,T]}|X(t)-Y(t)|>0\big)=0.$$  
Let us denote by $|\cdot|_F$ the Frobenius norm
$$
||\sigma||_F:=\sqrt{\tr(\sigma\sigma^T)}
$$
for all $\sigma \in  M^{n\times2n}$. Then by \eqref{p range}, it holds
\begin{align}\label{intro: stoc int well defined}
\mathbb E\int_0^T||\sigma\big(X(l),G(l)\big)||^2_F\ud l\leq 2T(p_{\text{max}}-2+n)<\infty.
\end{align}
Hence, the stochastic integral in the right-hand side of \eqref{stoc dynamics} is well defined. Furthermore, the functions $\rho$ and $\sigma$ are continuous with respect to the control parameters. Because the controls of the players are admissible, it holds 
\begin{align}\label{existence admissible}
\mathbb E\int_0^T\Big|\rho\big(G(s)\big)\Big|^2\ud s\leq \big(|\mu|+2(\Lambda(A_1)+\Lambda(A_2))\big)^2T<\infty
\end{align}
for $\Lambda(A_1),\Lambda(A_2)<\infty$, where $\Lambda(\cdot)$ is defined in \eqref{admissible controls}. Moreover, we can estimate
\begin{align*}
||\sigma\big(x,G(t)\big)-\sigma\big(y,G(t)\big)||_F&\leq \sqrt{2}\big|\sqrt{p(x,t)-1}-\sqrt{p(y,t)-1}\big|\\
&\leq \frac{|p(x,t)-p(y,t)|}{\sqrt{2p_{\text{min}}-2}} \\
&\leq \frac{L_p}{\sqrt{2p_{\text{min}}-2}}|x-y|
\end{align*}
for all $x,y\in \mathbb R^n$ and $t\in[0,T]$ with $L_p$ denoting the Lipschitz constant of $p$. Therefore by combining this, \eqref{intro: stoc int well defined}, \eqref{existence admissible} and \cite[Theorem 2.5.7]{krylov80}, the SDE \eqref{stoc dynamics} admits a pathwise unique strong solution.

Throughout, we denote by $||\cdot||$ a matrix norm
$$
||M||:=\sup_{|x|=1}\big|\langle Mx,x\rangle \big|
$$
for all $n \times n$ matrices $M$. Furthermore, $S(n)$ denotes the set of all symmetric $n\times n$ matrices, $I$ is the $n\times n$ identity matrix, and for $\xi \in \mathbb R^{n}$, we denote by $\xi \otimes \xi$ the $n\times n$ matrix for which $(\xi\otimes \xi)_{ij}=\xi_i\xi_j$. A function $\zeta:[0,\infty)\to [0,\infty)$ is said to be a modulus, if it is continuous, nondecreasing, and satisfies $\zeta(0) = 0$.

\subsection{Viscosity solutions to Bellman-Isaacs equations with uniformly bounded action sets}
We define $\Phi : \mathcal{CS} \times  \mathbb R^n \times [0,T]\times \mathbb R^n \times S(n)\to \mathbb R$ through
\begin{align*}
\Phi\big(a,b,c,d;(x,t),\nu,M\big)=&-\tr\Big(\mathcal A_{a,b}^{(x,t)}M\Big)-(c+d)\langle a+b,\nu\rangle  -\langle \mu, \nu \rangle,
\end{align*}
where
\begin{align}\label{alpha matrix}
\mathcal A_{a,b}^{(x,t)}:=\frac{1}{2}\big(p(x,t)-2\big)(a \otimes a+b\otimes b)+I.
\end{align}
Observe that the matrix $\mathcal A_{a,b}^{(x,t)}$ is symmetric with eigenvalues between the values
\begin{align}\label{iso lambda}
\lambda:=\min\{1,p_{\text{min}}-1\}\text{ and }\Lambda:=\max\{1,p_{\text{max}}-1\}.
\end{align}

Given $m\in\{1,2,\dots\}$, we let 
$$
\mathcal H_m:=\mathbb S^{n-1}\times [0,m],
$$
and define $F_m^-,F_m^+:  \mathbb R^n\times [0,T]\times \mathbb R \times\mathbb R^n \times \mathcal S(n) \to \mathbb R$ through
\begin{align*}
F_m^-\big((x,t),\xi,\nu,M\big)&= \inf_{(a,c)\in \mathcal H_m}\sup_{(b,d)\in \mathcal H_m} \Phi\big(a,b,c,d;(x,t),\nu,M\big)+r\xi, \\
F_m^+\big((x,t),\xi,\nu,M\big)&=\sup_{(b,d)\in \mathcal H_m} \inf_{(a,c)\in \mathcal H_m} \Phi\big(a,b,c,d;(x,t),\nu,M\big)+r\xi
\end{align*}
for $r\geq 0$. Let $g:\mathbb R^n \to \mathbb R$ be a positive bounded Lipschitz function such that
\begin{align}\label{g assumptions}
\sup_{x\in \mathbb R^n} g(x)+\sup_{x,y\in \mathbb R^n,x\not=y} \frac{|g(x)-g(y)|}{|x-y|}<L_g
\end{align}
for some $L_g<\infty$. We study terminal value problems
\begin{align}
\begin{split}\label{eq: bdd 2}
\begin{cases}
\partial_tu-F_m^-\big((x,t),u,Du,D^2u\big)=0~~&\text{in }\mathbb R^n\times (0,T), \\
u(x,T)=g(x)~~&\text{on }\mathbb R^n
\end{cases}
\end{split}
\end{align}
and
\begin{align}
\begin{split}\label{eq: bdd 1}
\begin{cases}
\partial_tu-F_m^+\big((x,t),u,Du,D^2u\big)=0~~&\text{in }\mathbb R^n\times (0,T), \\
u(x,T)=g(x)~~&\text{on }\mathbb R^n.
\end{cases}
\end{split}
\end{align}
A common notion of weak solutions to these equations is viscosity solutions. We only consider solutions $u$ which satisfy a linear growth condition
\begin{equation}\label{linear growth}
|u(x,t)|\leq c(1+|x|)
\end{equation}
for all $(x,t)\in \mathbb R^n\times [0,T]$ and for some $c<\infty$ independent of $x,t$. We prove that there exists a unique viscosity solution to the equation \eqref{eq: bdd 2} satisfying the condition \eqref{linear growth}. We omit the proof for \eqref{eq: bdd 1}, because it is analogous. The proofs are based on the comparison principle and Perron's method. 
\begin{definition}\label{def: belmann-i}

($i$) A lower semicontinuous function $\ol{u}_m:\mathbb R^n \times [0,T] \to \mathbb R$ is a viscosity supersolution to \eqref{eq: bdd 2}, if it satisfies \eqref{linear growth}, 
$$
\ol{u}_m(x,T)\geq g(x)
$$
for all $x \in \mathbb R^n$, and if the following holds. For all $(x_0,t_0)\in \mathbb R^n \times (0,T)$ and for all $\phi \in C^{2,1}\big(\mathbb R^n \times (0,T)\big)$ such that
\begin{itemize}
\item $\ol{u}_m(x_0,t_0)=\phi(x_0,t_0)$
\item $\ol{u}_m(x,t)>\phi(x,t)$ for all  $(x,t)\not=(x_0,t_0)$
\end{itemize}
it holds
$$
\partial_t\phi(x_0,t_0)-F_m^-\big((x_0,t_0),\phi(x_0,t_0),D\phi(x_0,t_0),D^2\phi(x_0,t_0)\big)\leq 0.
$$

($ii$) An upper semicontinuous function $\underline{u}_m:\mathbb R^n \times [0,T] \to \mathbb R$ is a viscosity subsolution to \eqref{eq: bdd 2}, if it satisfies \eqref{linear growth}, 
$$
\underline{u}_m(x,T)\leq g(x)
$$
for all $x \in \mathbb R^n$, and if the following holds. For all $(x_0,t_0)\in \mathbb R^n \times (0,T)$ and for all $\phi \in C^{2,1}\big(\mathbb R^n \times (0,T)\big)$ such that
\begin{itemize}
\item $\underline{u}_m(x_0,t_0)=\phi(x_0,t_0)$
\item $\underline{u}_m(x,t)<\phi(x,t)$ for all  $(x,t)\not=(x_0,t_0)$
\end{itemize}
it holds
$$
\partial_t\phi(x_0,t_0)-F_m^-\big((x_0,t_0),\phi(x_0,t_0),D\phi(x_0,t_0),D^2\phi(x_0,t_0)\big)\geq 0.
$$

($iii$) If a function $u_m :\mathbb R^n \times [0,T] \to \mathbb R$ is a viscosity supersolution and a subsolution to \eqref{eq: bdd 2}, then $u_m$ is a viscosity solution to \eqref{eq: bdd 2}.
\end{definition}
Observe that we require the growth condition \eqref{linear growth} as a standing assumption for viscosity super- and subsolutions. We start with the following lemma. 

\begin{lemma}\label{existence} 
Let $y\in \mathbb R^n$, $0<\eps<1$, and let $L_g$ be the constant in \eqref{g assumptions} for $g$. Then, the functions
$$
\ol{a}(x,t)=g(y)+\frac{A}{\eps^{1/2}}(T-t)+2L_g\big(|x-y|^2+\eps\big)^{1/2},
$$ 
$$
\underline{a}(x,t)=g(y)-\frac{A}{\eps^{1/2}}(T-t)-2L_g\big(|x-y|^2+\eps\big)^{1/2}
$$
are viscosity super- and subsolutions to \eqref{eq: bdd 2}, respectively, if we choose $A$, independent of $y,\eps$ and $m$, large enough. 
\begin{proof}
Because $g$ is Lipschitz continuous with \eqref{g assumptions}, we get
\begin{align*}
\underline{a}(x,T)&\leq g(x)\leq \ol{a}(x,T)
\end{align*}
for all $x\in \mathbb R^n$. Furthermore, $\underline{a}$ and $\ol{a}$ satisfy \eqref{linear growth}. First, we prove that $\ol{a}$ is a supersolution. To establish this, since $\ol{a}$ is a smooth function, we need to show that
$$
\partial_t\ol{a}(x,t)-F_m^-\big((x,t),\ol{a}(x,t),D\ol{a}(x,t),D^2\ol{a}(x,t)\big)\leq 0
$$
for all $(x,t)\in \mathbb R^n \times (0,T)$. Let $(x,t)\in \mathbb R^n \times (0,T)$. By a direct calculation, it holds
\begin{align*}
D\ol{a}(x,t)=2L_g\big(|x-y|^2+\eps\big)^{-1/2}(x-y)
\end{align*}
and
\begin{align*}
D^2\ol{a}(x,t)=2L_g\big(|x-y|^2+\eps\big)^{-1/2}\bigg(I-\frac{(x-y)\otimes(x-y)}{|x-y|^2+\eps}\bigg).
\end{align*}
Thus, we can estimate
\begin{align*}
&-\tr\Big(\mathcal A_{a,b}^{(x,t)}D^2\ol{a}(x,t)\Big)= 2L_g\big(|x-y|^2+\eps\big)^{-1/2}\cdot\\
&\bigg\{\tr\Big(\mathcal A_{a,b}^{(x,t)}\Big(\big(|x-y|^2+\eps\big)^{-1}(x-y)\otimes (x-y)-I\Big)\Big)\bigg\} \\
&\geq -2n\Lambda L_g\big(|x-y|^2+\eps\big)^{-1/2}
\end{align*}
for all $a,b \in \mathbb S^{n-1}$. Furthermore, we have $\partial_t\ol{a}(x,t)=-A\eps^{-1/2}$.

We can assume $x\not=y$, because otherwise the next term below is zero. It holds
\begin{align*}
&\inf_{(a,c)\in \mathcal H_m}\sup_{(b,d)\in \mathcal H_m}-(c+d)\big\langle a+b,D\ol{a}(x,t)\big\rangle \\
&\geq 2L_g\big(|x-y|^2+\eps\big)^{-1/2}\inf_{(a,c)\in \mathcal H_m}-c\big\langle a-(x-y)/|x-y|,x-y\big\rangle \\
& \geq 0.
\end{align*}
In addition, we can estimate
$$
\Big|\big\langle \mu,D\ol{a}(x,t)\big\rangle\Big|\leq 2L_g|\mu||x-y|\big(|x-y|^2+\eps\big)^{-1/2}\leq 2L_g|\mu|.
$$
By combining our estimates above, we have
\begin{align*}
&\partial_t\ol{a}(x,t)-F_m^-\big((x,t),\ol{a}(x,t),D\ol{a}(x,t),D^2\ol{a}(x,t)\big)\\
&\leq -A\eps^{-1/2}+ 2n\Lambda L_g\big(|x-y|^2+\eps\big)^{-1/2}+2L_g|\mu|-r\ol{a}(x,t) \\
&\leq \eps^{-1/2}\big(-A+2n\Lambda L_g\big)+2L_g|\mu|.
\end{align*}
Hence, if we choose
$$
A=4L_g\big (n\Lambda+|\mu|\big),
$$
we can conclude that $\ol{a}$ is a supersolution to \eqref{eq: bdd 2}.

The proof that $\underline{a}$ is a subsolution to \eqref{eq: bdd 2} is very similar to the above. We need to show that 
$$
\partial_t\underline{a}(x,t)-F_m^-\big((x,t),\underline{a}(x,t),D\underline{a}(x,t),D^2\underline{a}(x,t)\big)\geq 0.
$$
Observe that for $x\not=y$, we have this time 
\begin{align*}
&\inf_{(a,c)\in \mathcal H_m}\sup_{(b,d)\in \mathcal H_m}-(c+d)\big\langle a+b,D\ol{a}(x,t)\big\rangle \\
&\leq 2L_g\big(|x-y|^2+\eps\big)^{-1/2}\sup_{(b,d)\in \mathcal H_m}-d\big\langle (x-y)/|x-y|+b,x-y\big\rangle \\
& \leq 0
\end{align*}
by estimating the infimum instead of the supremum. Thus, by repeating the argument above, we have
\begin{align*}
&\partial_t\underline{a}(x,t)-F_m^-\big((x,t),\underline{a}(x,t),D\underline{a}(x,t),D^2\underline{a}(x,t)\big) \\
&\geq \eps^{-1/2}\big(A- 2n\Lambda L_g\big)-2L_g|\mu|-r\underline{a}(x,t).
\end{align*}
Recall the assumption \eqref{g assumptions} implying $-r\underline{a}(x,t)\geq -rL_g$. Therefore by adjusting the constant $A$ large enough, we can conclude that $\underline{a}$ is a subsolution to  \eqref{eq: bdd 2}.
\end{proof}
\end{lemma}  

A useful tool for us is the comparison principle.
\begin{lemma}\label{uniqueness}
Let $\underline{u}_m$ and $\ol{u}_m$ be continuous viscosity sub- and supersolutions to \eqref{eq: bdd 2} in the sense of Definition \ref{def: belmann-i}, respectively. Then, it holds
$$
\underline{u}_m(x,t)\leq \ol{u}_m(x,t)
$$
for all $(x,t)\in \mathbb R^n \times [0,T]$.
\end{lemma}
The proof of the comparison principle can be found from \cite{buckdahnl08}, see also \cite{gigagis91}. Now, Lemmas \ref{existence} and \ref{uniqueness} applied to Perron's method yield the following result.
\begin{proposition}\label{prop: unique sol}
There exists a unique viscosity solution $u_m$ to \eqref{eq: bdd 2} in the sense of Definition \ref{def: belmann-i}. 
\end{proposition}
Observe that by comparison with a sufficiently large constant, the unique solution $u_m$ to \eqref{eq: bdd 2} is not merely of linear growth \eqref{linear growth}. It is even bounded. 

\section{The SDG with uniformly bounded action sets}\label{sec: bounded controls}
In this section, we examine the game dynamics under uniform bounds on the action sets of the players. In particular, we prove that the unique solution to \eqref{eq: bdd 2} equals the lower value function of the game under the uniform bound. For the upper value function, the proof is similar.
\begin{definition}
Let $\mathcal{AC}$ be the set of admissible controls, and let $\mathcal{S}$ be the set of admissible strategies in the sense of Definitions \ref{def:admissible controls} and \ref{def:admissible strategies}, respectively. For $m\in \{1,2,\dots\}$, we set
\begin{align*}
\mathcal{AC}_m&:=\{A\in\mathcal{AC}:\Lambda(A)\leq m\}, \\
\mathcal{S}_m&:=\{S\in\mathcal{S}:\Lambda(S)\leq m\},
\end{align*}
where $\Lambda(\cdot)$ is defined in \eqref{admissible controls} and \eqref{admissible strategy}.
\end{definition}

Let $m\in \{1,2,\dots\}$, and assume that the players choose their controls and strategies from the sets $\mathcal{AC}_m$ and $\mathcal S_m$, respectively. As before, the SDE \eqref{stoc dynamics} admits a pathwise unique strong solution. We define the lower and upper value functions of the game with controls in $\mathcal{AC}_m$ and strategies in $\mathcal{S}_m$ by setting
\begin{align}
\begin{split}\label{value func bounded controls}
U_m^-(x,t)&=\inf_{S\in \mathcal{S}_m}\sup_{A\in \mathcal{AC}_m} \mathbb E \Big[e^{-r(T-t)}g\big(X(T)\big)\Big], \\
U_m^+(x,t)&=\sup_{S\in \mathcal{S}_m}\inf_{A\in \mathcal{AC}_m} \mathbb E \Big[e^{-r(T-t)}g\big(X(T)\big)\Big]
\end{split}
\end{align}
for all $(x,t)\in \mathbb R^n\times [0,T]$, where $g$ is the pay-off \eqref{g assumptions}. The game starts at $x$ at a time $t$, and the expectation $\mathbb E$ is taken with respect to the measure $\mathbb P$. 

In Lemma \ref{main lemma add reg} below, we assume that the solution $u_m$ to \eqref{eq: bdd 2} is twice differentiable and that the solution and its derivatives of first and second order are Lipschitz continuous. Hence, we first study the so called sup- and inf-convolutions of the function $u_m$. In particular, for a large $j\in \mathbb N$, let us denote $T_j:=T-j^{-1}$ and $R_j^n:=\mathbb R^n\times [j^{-1},T_j]$. Then for $j$ fixed and $\eps>0$ small, we define
$$
u_\eps(x,t)=\sup_{(z,s)\in\mathbb R^n\times[0,T]}\Big(u_m(z,s)-\frac{1}{2\eps}\big((t-s)^2+|x-z|^2\big)\Big)
$$
whenever $(x,t)\in R_j^n$. The sup-convolution $u_\eps$ has well-known properties. Indeed, $u_\eps$ is locally Lipschitz continuous, semiconvex and $u_\eps \searrow u_m$ as $\eps \to 0$, see for example \cite{grandallil92}. Moreover, $u_\eps$ yields a good approximation of $u_m$ in the viscosity sense. The proof of the following lemma follows \cite{ishii95}, where they consider an elliptic case. For the benefit of the reader, we give the proof in our parabolic setting.

\begin{lemma}\label{use ishii app}
Let $u_m$ be a viscosity solution to \eqref{eq: bdd 2}, and let $u_\eps$ be the sup-convolution of $u_m$. Then for $\eps$ small enough, it holds
\begin{align*}
F_m^-\big((x,t),u_\eps(x,t),Du_\eps(x,t),D^2u_\eps(x,t)\big)\leq \partial_tu_\eps(x,t)+\zeta(\eps)
\end{align*}
for a.e.\ $(x,t)\in R_j^n$ with a bounded modulus of continuity $\zeta(\eps)$.
\begin{proof}
By the comparison principle and the assumption \eqref{g assumptions} on $g$, it holds $0\leq u_m\leq L_g$. Therefore for all $(x,t)\in R_j^n$ and $\eps>0$ small enough, there exists a point $(x^*,t^*)\in \mathbb R^n\times ]0,T[$, where the supremum used in the definition of $u_\eps$ is obtained. In particular, it holds
\begin{align*}
0\leq u_m(x,t)\leq u_\eps(x,t)\leq L_g-\frac{1}{2\eps}\big((t-t^*)^2+|x-x^*|^2\big).
\end{align*}
Hence, this yields 
$
|t-t^*|<j^{-1},
$
if $\eps<1/(2L_gj^2)$.

By the Lipschitz continuity and the semiconvexity of $u_\eps$, it holds
\begin{align}
\begin{split}\label{jensen}
u_\eps(z,s)\leq~& u_\eps(x,t)+\partial_t u_\eps(x,t)(s-t)+\langle Du_\eps(x,t),z-x\rangle\\
&+\frac{1}{2}\big\langle D^2u_\eps(x,t)(z-x),z-x\big\rangle +  o\big(|s-t|+|z-x|^2\big)
\end{split}
\end{align}
for a.e.\ $(x,t)\in R_j^n$ as $(z,s)\to (x,t)$, see \cite[Lemmas 3.3 and 3.15]{jensen88}. Here, we also applied the fundamental Aleksandrov's theorem for convex functions, see for example \cite[Theorem 6.4.1]{evansg92}. Moreover, the estimate \eqref{jensen} implies that we can choose $(x^*,t^*)$ such that 
\begin{align}
\begin{split}\label{magic property}
x^*&=x+\eps Du_\eps(x,t),\\
t^*&=t+\eps\partial_tu_\eps(x,t)
\end{split}
\end{align}
for a.e.\ $(x,t)\in R_j^n$, see \cite[Lemma A.5]{grandallil92} or \cite[Theorem 4.7]{katzourakis15}. Let $(x,t)\in R_j^n$ such that \eqref{jensen} holds.  We define $v:R_j^n \to \mathbb R$ through
\begin{align*}
v(z,s)=&~\partial_t u_\eps(x,t)(s-t)+\langle Du_\eps(x,t),z-x\rangle \\
&+\frac{1}{2}\big\langle D^2u_\eps(x,t)(z-x),z-x\big\rangle
\end{align*}
for $(z,s)\in R_j^n$. We want to find a local maximum of a function at $(x^*,t^*,x,t)$ up to an error in order to use the parabolic theorem of sums. Because it holds $v(x,t)=0$ and 
$$
u_m(y,l)-\frac{1}{2\eps}\big((l-s)^2+|y-z|^2\big)\leq u_\eps(z,s)
$$
for all $(z,s),(y,l)\in R_j^n$, we can estimate by \eqref{jensen}
\begin{align*}
&u_m(y,l)-v(z,s)-\frac{1}{2\eps}\big((l-s)^2+|y-z|^2\big)\\
&\leq u_m(x^*,t^*)-v(x,t)-\frac{1}{2\eps}\big((t-t^*)^2+|x-x^*|^2\big)\\
&\hspace{1em}+o\big(|s-t|+|z-x|^2\big)
\end{align*}
for any $(y,l)\in R_j^n$ as $(z,s)\to (x,t)$. By using this inequality, we can deduce
\begin{align}
\begin{split}\label{eq: enough for thm of sums A}
&u_m(y,l)-v(z,s) \\
&\leq u_m(x^*,t^*)-v(x,t)+\frac{1}{\eps}\langle x^*-x,y-x^*\rangle+\frac{1}{\eps}(t^*-t)(l-t^*) \\
&\hspace{1em}+\frac{1}{\eps}\langle x-x^*,z-x\rangle+\frac{1}{\eps}(t-t^*)(s-t)+\frac{1}{2\eps}\big(|y-x^*|^2+|z-x|^2\big) \\
&\hspace{1em}-\frac{1}{\eps}\langle y-x^*,z-x\rangle+o\big(|s-t|+|l-t^*|+|z-x|^2\big)
\end{split}
\end{align}
for all $y\in \mathbb R^n$ as $(z,s,l)\to(x,t,t^*)$. This is true, because by direct calculations it holds
\begin{align*}
&\frac{1}{2\eps}\big((l-s)^2-(t-t^*)^2\big)\\
&=\frac{1}{2\eps}\Big(\big(t-s+l-t^*\big)^2-2(t^*-t)^2+2(t^*-t)(l-s)\Big) \\
&\leq\frac1\eps(t^*-t)(l-t^*)+\frac1\eps(t-t^*)(s-t)+o\big(|s-t|+|l-t^*|\big)
\end{align*}
as $(s,l)\to(t,t^*)$ and
\begin{align*}
&\langle x^*-x,y-x^*\rangle+\langle x-x^*,z-x\rangle+\frac{1}{2}\big(|y-x^*|^2+|z-x|^2\big)-\langle y-x^*,z-x\rangle \\
&=\frac{1}{2}\big(|y-z|^2+|x-x^*|^2\big)
\end{align*}
for all $y,z\in \mathbb R^n$.

For the following notation and use of the parabolic theorem of sums, we refer the reader to \cite{grandallil92}, see also \cite{katzourakis15}. By the estimate \eqref{eq: enough for thm of sums A}, it holds
\begin{align*}
&\bigg(\frac1\eps(x^*-x),\frac1\eps(t^*-t),\frac1\eps(x-x^*),\frac1\eps(t-t^*),\frac1\eps\begin{bmatrix}
    I & -I \\
    -I& I
  \end{bmatrix}\bigg)\\
&\in\mathcal P^{2,+}\Big(u_m(x^*,t^*)-v(x,t)\Big).
\end{align*}
Thus by \cite[Theorem 6.7]{katzourakis15}, there exist symmetric matrices $Y:=Y(\eps)$ and $Z:=Z(\eps)$ such that
\begin{align*}
&\bigg(\frac{1}{\eps}(t^*-t),\frac{1}{\eps}(x^*-x),Y\bigg)\in \mathcal{\ol{P}}^{2,+}u_m(x^*,t^*) \\
&\bigg(\frac{1}{\eps}(t^*-t),\frac{1}{\eps}(x^*-x),Z\bigg)\in \mathcal{\ol{P}}^{2,-}v(x,t) 
\end{align*}
and
\begin{align}\label{thm of sums1}
\begin{bmatrix}
    Y & 0 \\
   0&-Z
  \end{bmatrix}\leq \frac{3}{\eps}  \begin{bmatrix}
    I & -I\\
    -I& I
  \end{bmatrix}.
\end{align}
Therefore, because $u_m$ is a subsolution, this and \eqref{magic property} yield
\begin{align}\label{eq: estimate I1}
F_m^-\big((x^*,t^*),u_m(x^*,t^*),Du_\eps(x,t),Y\big)\leq \partial_tu_\eps(x,t).
\end{align}
Furthermore, since $D^2v(x,t)=D^2u_\eps(x,t)$, the degenerate ellipticity of $F_m^-$ implies 
\begin{align*}
&F_m^-\big((x,t),u_\eps(x,t),Du_\eps(x,t),D^2u_\eps(x,t)\big) \\
&\leq F_m^-\big((x,t),u_\eps(x,t),Du_\eps(x,t),Z\big).
\end{align*}
By combining this and \eqref{eq: estimate I1}, the proof is complete, if we can show that there exists a modulus $\zeta$ such that 
\begin{align}
\begin{split}\label{final part to show approx}
&F_m^-\big((x,t),u_\eps(x,t),Du_\eps(x,t),Z\big) \\
&\leq F_m^-\big((x^*,t^*),u_m(x^*,t^*),Du_\eps(x,t),Y\big)+ \zeta(\eps).
\end{split}
\end{align}
We prove this inequality by utilizing \eqref{thm of sums1}.

Let $a,b\in \mathbb S^{n-1}$. We multiply from the left both sides in \eqref{thm of sums1} by
$$
\begin{bmatrix}
\mathcal A_{a,b}^{(x^*,t^*)} &\mathcal A_{a,b}^{(x,t),(x^*,t^*)} \\
\mathcal A_{a,b}^{(x,t),(x^*,t^*)}&\mathcal A_{a,b}^{(x,t)}
\end{bmatrix},
$$
where
$$
\mathcal A_{a,b}^{(x,t),(x^*,t^*)}:=\frac{1}{2}\Big(\sqrt{p(x^*,t^*)-1}\sqrt{p(x,t)-1}-1\Big)(a\otimes a+ b\otimes b)+I,
$$
and the matrices $\mathcal A_{a,b}^{(x,t)}$ and $\mathcal A_{a,b}^{(x^*,t^*)}$ are defined in \eqref{alpha matrix}.
Then by taking traces and observing
$$
\tr(a\otimes a+b\otimes b)=2,
$$
we get
\begin{align}
\begin{split}\label{tr estimate A}
&-\tr\big(\mathcal A_{a,b}^{(x,t)}Z\big)+\tr\big(\mathcal A_{a,b}^{(x^*,t^*)}Y\big) \\
&\leq\frac{3}{\eps}\bigg(\tr\big(\mathcal A_{a,b}^{(x^*,t^*)}+\mathcal A_{a,b}^{(x,t)}\big)-2\tr\big(\mathcal A_{a,b}^{(x,t),(x^*,t^*)}\big)\bigg) \\
&=\frac{3}{\eps}\Big(\sqrt{p(x,t)-1}-\sqrt{p(x^*,t^*)-1}\Big)^2.
\end{split}
\end{align}
Because it holds $p_{\text{min}}>1$ and 
$$
\sqrt f-\sqrt h=\frac{(\sqrt f+\sqrt h)(\sqrt f-\sqrt h)}{\sqrt f+\sqrt h}=\frac{f-h}{\sqrt f+\sqrt h}
$$
for any $f,h>0$, we can estimate
\begin{align*}
&\frac{3}{\eps}\Big(\sqrt{p(x,t)-1}-\sqrt{p(x^*,t^*)-1}\Big)^2 \leq \frac{3L_p^2}{2(p_{\text{min}}-1)}\cdot\frac{1}{2\eps}\big((t-t^*)^2+|x-x^*|^2\big)
\end{align*}
with $L_p$ denoting the Lipschitz constant of $p$. Therefore, because $\mathcal H_m$ is compact, $\Phi$ is continuous with respect to the variables in $\mathcal{CS}$ and $a,b$ are arbitrary, this and \eqref{tr estimate A} imply 
\begin{align*}
&F_m^-\big((x,t),u_\eps(x,t),Du_\eps(x,t),Z\big)- F_m^-\big((x^*,t^*),u_m(x^*,t^*),Du_\eps(x,t),Y\big) \\
&\leq\frac{3L_p^2}{2(p_{\text{min}}-1)}\cdot\frac{1}{2\eps}\big((t-t^*)^2+|x-x^*|^2\big).
\end{align*}
The solution $u_m$ is H\"{o}lder continuous, see Lemma \ref{lemma: equicontinuity} below. In particular, there exists a modulus $\zeta_u$, independent of $m$, such that
\begin{align*}
\frac{1}{2\eps}\big((t-t^*)^2+|x-x^*|^2\big) \leq u_m(x^*,t^*)-u_m(x,t)\leq \zeta_u\big(\sqrt{2L_g\eps}\big).
\end{align*}
Thus by denoting 
$$\zeta(\eps):=\frac{3L_p^2}{2(p_{\text{min}}-1)}\zeta_u\big(\sqrt{2L_g\eps}\big)$$ 
and recalling \eqref{final part to show approx}, the proof is complete.
\end{proof}
\end{lemma}

We prove the following main lemma of this section.
\begin{lemma}\label{lemma: game value}
Let $u_m$ be the unique viscosity solution to the equation \eqref{eq: bdd 2}. Furthermore, let $U_m^-$ be the lower value function of the game defined in \eqref{value func bounded controls}. Then, it holds
\begin{align*}
u_m(x,t)&=U_m^-(x,t) 
\end{align*}
for all $(x,t)\in \mathbb R^n \times [0,T]$.
\end{lemma}
\begin{proof}
To establish the result, we regularize the solution $u_m$ first by the sup-convolution and then by the standard mollification. Then, we apply Lemma \ref{main lemma add reg} below to the regularized function and finally pass to the limits.

Fix a large $j\in \mathbb N$ and a small $\eps>0$. By Lemma \ref{use ishii app}, it holds
\begin{align}\label{eq: ishii}
F_m^-\big((x,t),u_\eps(x,t),Du_\eps(x,t),D^2u_\eps(x,t)\big)\leq \partial_tu_\eps(x,t)+\zeta(\eps)
\end{align}
for a.e.\ $(x,t)\in R_j^n$ with a bounded modulus of continuity $\zeta(\eps)$. Let $\delta>0$ be small, and denote by $\phi_\delta$ the standard mollifier in $\mathbb R^{n+1}$. Then for $\delta$ small enough, the function $u_\eps^\delta:=\phi_\delta*u_\eps$ is well defined on $R_{j-1}^n$. Because $u_\eps$ is bounded, the mollification ensures that $u_\eps^\delta$ is bounded uniformly in $\delta$, and $u_\eps^\delta$ is Lipschitz continuous. Moreover, $u_\eps^\delta$ is smooth, and $Du_\eps^\delta, \partial_t u_\eps^\delta$ and $D^2u_\eps^\delta$ are bounded and Lipschitz continuous on $R_{j-1}^n$. In addition, because $u_\eps$ is continuous on $R_j^n$, it holds that $u_\eps^\delta \to u_\eps$ uniformly as $\delta \to 0$ on $R_{j-1}^n$. We can also show that it holds
\begin{align*}
Du_\eps^\delta(x,t)&\to Du_\eps(x,t), \\
\partial_t u_\eps^\delta(x,t) &\to \partial_t u_\eps(x,t),\\
D^2u_\eps^\delta(x,t) &\to D^2u_\eps(x,t)
\end{align*}
as $\delta \to 0$ for a.e.\ $(x,t)\in R^n_{j-1}$, see for example \cite{evansg92}. Furthermore, we have
$$
F_m^-\big((x,t),u_\eps^\delta(x,t),Du^\delta_\eps(x,t),D^2u^\delta_\eps(x,t)\big)\leq \partial_tu^\delta_\eps(x,t)+\zeta(\eps)+\gamma_\delta(x,t)
$$
for all $(x,t)\in R_{j-1}^n$, where it holds 
\begin{align*}
\gamma_\delta(x,t)&:=\max\Big\{F_m^-\big((x,t),u_\eps^\delta(x,t),Du^\delta_\eps(x,t),D^2u^\delta_\eps(x,t)\big)-\partial_tu^\delta_\eps(x,t),\zeta(\eps)\Big\}\\
&\hspace{1.5em}-\zeta(\eps).
\end{align*}
By using the convergences above and \eqref{eq: ishii}, we see $\gamma_\delta\to 0$ as $\delta\to0$ for a.e.\ on $R_{j-1}^n$. It also holds that $\gamma_\delta$ is uniformly continuous on $R_{j-1}^n$ and bounded from above uniformly with respect to $\delta$. This is true, because the operator $F_m^-$ and the variables are uniformly continuous, and $u_\eps^\delta$ is uniformly Lipschitz and semiconvex with respect to $\delta$. Now by doing minor adjustments to the proof of Lemma \ref{main lemma add reg} below, we can argue that
\begin{align}
\begin{split}\label{eq: estimate A1}
u_\eps^\delta(x,t)\leq \inf_{S\in \mathcal S_m}\sup_{A\in \mathcal{AC}_m}\mathbb E\bigg[&\int_t^{T_{j-1}}e^{-r(l-t)}h_\eps^\delta\big(X(l),l\big)\ud l\\
&+e^{-r(T_{j-1}-t)}u_\eps^\delta\big(X(T_{j-1}),T_{j-1}\big)\bigg]
\end{split}
\end{align}
for all $(x,t)\in R_{j-1}^n$ with $h_\eps^\delta:=\zeta(\eps)+\gamma_\delta$ and $\eps$ small enough. This is true, because $h_\eps^\delta$ is uniformly continuous. 

Next, for $j$ fixed, we let $\delta \to 0$ and $\eps \to 0$. First, we make a rough estimate for the drift part and apply Doob's martingale inequality for the diffusion part of the process $\big(X(l)\big)$ to get the following. For all $\theta>0$, we choose $R:=R(\theta,m,\mu,n,p_{\text{max}},T)>0$, independent of controls and strategies, large enough such that
$$
\mathbb P\Big(\sup_{t\leq l \leq T}\big|X(l)-x\big|\geq R\Big)\leq \theta,
$$ 
see for example \cite[Theorem 2.7.2.2]{evans13}. Then by Egorov's theorem, we find a set $U_\theta\subset B_R(x) \times  [0,T]$ such that $|U_\theta|\leq \theta$ and 
\begin{align}\label{gamma to zero}
\gamma_\delta\to 0\text{ uniformly as }\delta \to 0\text{ on }\big(B_R(x)\times [(j-1)^{-1},T_{j-1}]\big)\setminus U_\theta. 
\end{align}
Now, we estimate
\begin{align}
\begin{split}\label{eq: estimate A2}
\mathbb E\int_t^{T_{j-1}}e^{-r(l-t)}h_\eps^\delta\big(X(l),l\big)\ud l\leq ~&I_1^{\eps,\delta}(\theta)+I_2^{\eps,\delta}(\theta)\\
&+\big(C_\gamma+\zeta(\eps)\big)(T_{j-1}-t)\theta,
\end{split}
\end{align}
where we denoted by $C_\gamma<\infty$ a constant such that $\sup_{R^n_{j-1}} \gamma_\delta <C_\gamma$ and
\begin{align*}
I_1^{\eps,\delta}(\theta)&:=\mathbb E\int_t^{T_{j-1}}e^{-r(l-t)}h_\eps^\delta\big(X(l),l\big)\chi_{U_\theta}\big(X(l),l\big) \ud l, \\
I_2^{\eps,\delta}(\theta)&:=\mathbb E\int_t^{T_{j-1}}e^{-r(l-t)}h_\eps^\delta\big(X(l),l\big)\chi_{\big(B_R(x)\times [t,T_{j-1}]\big)\setminus U_\theta}\big(X(l),l\big) \ud l.
\end{align*}
By a fundamental estimate in \cite[Theorem 3.4]{krylov80}, see also \cite{krylovs79}, it holds
$$
\mathbb E\int_t^{T_{j-1}}\Big[e^{-r(l-t)}\chi_{U_\theta}\big(X(l),l\big)\Big]\ud l\leq C(T_{j-1}-t)|U_\theta|
$$
for a constant $C:=C(n,p_{\text{min}},p_{\text{max}},m,\mu,r)<\infty$. Hence, we have
\begin{align}\label{est I2}
I_1^{\eps,\delta}(\theta)\leq C(T_{j-1}-t)\theta\big( C_\gamma+\zeta(\eps)\big).
\end{align}
Furthermore, because we have \eqref{gamma to zero} and $\zeta(\eps)\to 0$  as $\eps \to 0$, it holds $I_2^{\eps,\delta}(\theta)\to 0$  by first letting $\delta \to 0$ and then $\eps \to 0$.

Combining this together with the estimates \eqref{eq: estimate A1}, \eqref{eq: estimate A2} and \eqref{est I2}, and letting $\delta,\theta,\eps \to 0$, we have proven 
\begin{align*}
u_m(x,t)\leq\inf_{S\in \mathcal S_m}\sup_{A\in \mathcal{AC}_m}\mathbb E \Big[e^{-r(T_{j-1}-t)}u_m\big(X(T_{j-1}),T_{j-1}\big)\Big]
\end{align*}
for all $(x,t)\in R_{j-1}^n$. Finally by recalling $T_{j-1}=T-(j-1)^{-1}$ and letting $j\to \infty$, we see by utilizing the barrier constructed in Lemma \ref{existence} that
\begin{align}\label{final result A1}
u_m(x,t)\leq\inf_{S\in \mathcal S_m}\sup_{A\in \mathcal{AC}_m}\mathbb E \Big[e^{-r(T-t)}g\big(X(T)\big)\Big].
\end{align}
Here, we also applied Jensen's inequality, Ito's isometry and \eqref{comp stoc dynamics} to get
\begin{align*}
&\mathbb E\Big(|X(T_{j-1})-X(T)|^2+j^{-1}\Big)^{1/2} \leq \Big(\mathbb E|X(T_{j-1})-X(T)|^2+j^{-1}\Big)^{1/2} \\
&\leq \Big(C(j-1)^{-1}+j^{-1}\Big)^{1/2}
\end{align*}
with a constant $C:=C(m,\mu,n,p_{\text{max}})<\infty$ to estimate terms in the barrier.

The proof of the opposite inequality in \eqref{final result A1} is analogous. In particular, we first apply the inf-convolution 
$$
\tilde{u}_\eps(x,t)=\inf_{(z,s)\in\mathbb R^n\times[0,T]}\Big(u_m(z,s)+\frac{1}{2\eps}\big((t-s)^2+|x-z|^2\big)\Big)
$$
whenever $(x,t)\in R_j^n$, and deduce an opposite type of inequality similar to \eqref{eq: ishii} with the same modulus of continuity $\zeta$. Then, we make the standard mollification, and deduce the result by passing to the limits as before. Therefore, the proof is complete. 
\end{proof}
In the result above, we utilized the following two lemmas.
\begin{lemma}\label{lemma: phi cont}
Let $u:\mathbb R^n \to \mathbb R$ be twice differentiable, and let $a,b\in \mathbb S^{n-1}$ and $c,d\in[0,m]$ with $m\in \mathbb N$. Furthermore, assume that $Du$ and $D^2 u$ are Lipschitz continuous, and $D^2u$ is bounded. Then, the function
$$
(x,t) \mapsto \Phi\big(a,b,c,d;(x,t),Du(x,t),D^2u(x,t)\big)
$$
is also Lipschitz continuous.
\begin{proof}
By a direct computation, it holds
\begin{align}
\begin{split}\label{eq: estimate U1}
&\big\langle (c+d)(a+b)+\mu,Du(x,t)-Du(z,s)\big\rangle\\
&\hspace{1em}+\tr\Big[D^2u(x,t)-D^2u(z,s)\Big] \\
&\leq L\big(|x-z|^2+(t-s)^2\big)^{1/2}
\end{split}
\end{align}
for all $(x,t),(z,s)\in\mathbb R^n\times [0,T]$ and for a constant $L:=L(m,\mu,n,L_1,L_2)$ with $L_1$ denoting the Lipschitz constant of $Du$ and $L_2$ denoting the Lipschitz constant of $D^2u$, respectively. Furthermore, because $D^2u$ is bounded, we have
$$
C_0:=\sup_{(z,l)\in\mathbb R^n\times [0,T]}\big|\big|D^2u(z,l)\big|\big|<\infty.
$$
Therefore, we can estimate
\begin{align*}
&\big(p(x,t)-2\big)\tr\Big( (a \otimes a+b\otimes b) D^2u(x,t)\Big)\\
&\hspace{1em}-\big(p(z,l)-2\big)\tr\Big( (a \otimes a+b\otimes b) D^2u(z,l)\Big) \\
&=\big(p(x,t)-2\big)\tr\Big( (a \otimes a+b\otimes b)\big( D^2u(x,t)-D^2u(z,l)\big)\Big) \\
&\hspace{15pt}+\big(p(x,t)-p(z,l)\big)\tr\Big( (a \otimes a+b\otimes b) D^2u(z,l)\Big) \\
&\leq \tilde{L}\big(|x-z|^2+(t-s)^2\big)^{1/2}
\end{align*}
for all $(x,t),(z,s)\in\mathbb R^n\times [0,T]$ and for a constant $\tilde{L}:=(p_{\text{max}},n,L_2,L_p,C_0)$ with $L_p$ denoting the Lipschitz constant of $p$. Thus, this estimate, together with the estimate \eqref{eq: estimate U1}, completes the proof. 

\end{proof}
\end{lemma}

\begin{lemma}\label{main lemma add reg}
Let $u_m$ be the unique viscosity solution to the equation \eqref{eq: bdd 2}, and let $U_m^-$ be the lower value function of the game defined in \eqref{value func bounded controls}. Furthermore, assume that $u_m$ is twice differentiable such that $u_m$, $\partial_t u_m$, $Du_m$, $D^2u_m$ are Lipschitz continuous, and $Du_m$, $D^2u_m$ are bounded in $\mathbb R^n\times [0,T)$. Then, it holds
\begin{align*}
u_m(x,t)&=U_m^-(x,t) 
\end{align*}
for all $(x,t)\in \mathbb R^n \times [0,T]$.
\begin{proof}The idea of the proof is to apply Ito's formula to connect the solution $u_m$ and the lower value function $U_m$ with uniformly bounded controls. We utilize discretized controls based on the solution $u_m$, and in the end, pass to a limit with the discretization parameter.

Let $k \in \mathbb N$ be an integer, $(x,t)\in \mathbb R^n\times [0,T)$ and denote $\triangle t:=(T-t)/k$ and $t_i:=t+i\triangle t$ for all $i\in \{0,\dots,k\}$. Note that $t_0=t$ and $t_k=T$, and set $E_i:=[t_{i-1},t_i)$ for all $i\in\{1,\dots,k\}$. For the time interval $E_1$, we can choose a constant control $(a^1,c_1)\in \mathcal H_m$ such that
\begin{align}
\begin{split}\label{chosen control}
\sup_{(b,d)\in\mathcal H_m}&\Phi\big(a^1,b,c_1,d,Du_m(x,t),D^2u_m(x,t)\big)+ru_m(x,t) \\
&\leq \partial_t u_m(x,t)+\frac{1}{k},
\end{split}
\end{align}
since $u_m$ is a solution to \eqref{eq: bdd 1}. Let $s\in E_1$, and let $\big\{\big(b(l),d(l)\big)\big\}\in \mathcal{AC}_m$ be an arbitrary control. We define $X(s)$ as in \eqref{stoc dynamics} with $X(t)=x$ and controls $(a^1,c_1)$ and $\big(b(l),d(l)\big)$, $l\in [t,s]$. By the assumptions, $u_m$ is regular enough to utilize Ito's formula. Thus, it holds
\begin{align}
\begin{split}\label{ito to u}
&u_m\big(X(s),s\big)-u_m(x,t) \\
&=\int_t^s\partial_t u_m\big(X(l),l)\ud l +\sum_{i=1}^n\int_t^s \frac{\partial u_m}{\partial x_i}\big(X(l),l\big)\ud X_i(l) \\
&\hspace{1em}+\frac{1}{2}\sum_{i,j=1}^n\int_t^s\frac{\partial^2 u_m}{\partial x_i\partial x_j}(X(l),l)\ud \langle X_i,X_j\rangle(l).
\end{split}
\end{align}
For brevity, we denote
\begin{align*}
\Phi_1^X(s)&:=\Phi\big(a^1,b(s),c_1,d(s);(X(s),s),Du_m(X(s),s),D^2u_m(X(s),s)\big), \\
\Phi_1^x(s)&:=\Phi\big(a^1,b(s),c_1,d(s);(x,s),Du_m(x,s),D^2u_m(x,s)\big).
\end{align*}
Therefore by utilizing \eqref{comp stoc dynamics} and \eqref{ito to u}, we get
\begin{align}
\begin{split}\label{ito final}
u_m\big(X(s),s\big)&=u_m(x,t)+\int_t^s\big(\partial_t u_m\big(X(l),l)-\Phi_1^X(l)\big)\ud l\\
&\hspace{1em}+G\big(X(s),s\big).
\end{split}
\end{align}
Here, it holds
\begin{align*}
G\big(X(s),s\big)=&\sum_{i=2}^n\bigg(\int_t^s\big\langle Du_m(X(l),l),p_{a^1}^{i-1}\big\rangle\ud W_i^1(l) \\
&\hspace{4em}+\int_t^s\big\langle Du_m(X(l),l),p_{b(l)}^{i-1}\big\rangle\ud W_i^2(l)\bigg) \\
&+\sqrt{p\big(X(l),l\big)-1}\bigg(\int_t^s\big\langle Du_m(X(l),l),a^1\big\rangle\ud W_1^1(l) \\
&\hspace{9em}+\int_t^s\big\langle Du_m(X(l),l),b(l)\big\rangle\ud W_2^1(l)\bigg),
\end{align*}
where we recall that $p_{\nu}^i$ denotes the $i$-th column vector of the matrix $P^\bot_{\nu}$ for all $\nu\in \mathbb S^{n-1}$.

We note that for any adapted one dimensional process $\big\{\theta(l)\big\}_{l\in[0,T]}$ with $\mathbb E\int_0^T\theta^2(l)\ud l<\infty$, it holds
$$
\mathbb E\int_0^h\theta(l)\ud W(l)=0
$$
for all $h\in[0,T]$, where $W$ is a one dimensional Brownian motion starting from the origin. Thus, because $Du_m$ and $p$ are assumed to be bounded, it holds $$\mathbb E G\big(X(s),s\big)=0.$$ Therefore by estimating the function $(z,l) \mapsto e^{-rl}u_m(z,l)$ instead of $(z,l) \mapsto u_m(z,l)$ in a similar way to \eqref{ito final}, it holds
\begin{align*}
&\mathbb E \big[e^{-rs}u_m\big(X(s),s\big)-e^{-rt}u_m(x,t)\big] \\
& =\mathbb E \int_t^s e^{-rl}\big(\partial_t u_m\big(X(l),l\big)-\Phi_1^X(l)-ru_m\big(X(l),l\big)\big)\ud l.
\end{align*}
This implies
\begin{align*}
u_m(x,t)=\mathbb E \Big[&e^{-r(s-t)}u_m\big(X(s),s\big) \\
&- \int_t^s e^{-r(l-t)}\big(\partial_t u_m\big(X(l),l)-\Phi_1^X(l)-ru_m\big(X(l),l\big)\big)\ud l\Big].
\end{align*}
Next, we add and subtract terms so that we can utilize \eqref{chosen control}. In particular, it holds
\begin{align}
\begin{split}\label{eq:A}
u_m(x,t)=\mathbb E \Big[&e^{-r(s-t)}u_m\big(X(s),s\big) +K_1+K_2+K_3 \\
&+ \int_t^s e^{-r(l-t)}\big(-\partial_t u_m(x,t)+\Phi_1^x(t)+ ru_m(x,t)\big)\ud l\Big],
\end{split}
\end{align}
where
\begin{align*}
K_1&= \int_t^s e^{-r(l-t)}\big(\partial_t u_m(x,t)- \partial_t u_m\big(X(l),l\big)\big)\ud l,\\ 
K_2&=\int_t^s e^{-r(l-t)}\big(\Phi_1^X(l)- \Phi_1^x(t)\big)\ud l,\\
K_3&=\int_t^s e^{-r(l-t)}\big(ru_m\big(X(l),l\big)- ru_m(x,t)\big)\ud l.
\end{align*}
Hence by using \eqref{chosen control} to estimate the last term in \eqref{eq:A}, we get
\begin{equation}\label{eq:B}
u_m(x,t) \leq \mathbb E \Big[e^{-r(s-t)}u_m\big(X(s),s\big) +K_1+K_2+K_3\Big]+\frac{s-t}{k}. 
\end{equation}

We recall that $u_m$, $\partial_t u_m$, $Du_m$, and $D^2u_m$ are Lipschitz continuous, and we denote the largest Lipschitz constant of these by $L_m$. Then, we can estimate
$$
\mathbb E|K_1|+\mathbb E|K_3| \leq(1+r)L_m\Big[(s-t)^2+\mathbb E\int_t^s|X(l)-x|\ud l\Big].
$$
Furthermore, let us denote 
$$
C_{0,m}:=\sup_{(z,l)\in\mathbb R^n\times [0,T]}\big|\big|D^2u_m(z,l)\big|\big|,
$$
which is assumed to be bounded. Then, Lemma \ref{lemma: phi cont} yields
\begin{align*}
\big|\Phi_1^X(l)- \Phi_1^x(t)\big|\leq L\big(|X(l)-x|^2+(s-t)^2\big)^{1/2}
\end{align*}
for all $l\in[t,s]$ and for a constant $L:=L(m,\mu,p_{\text{max}},n,L_m,C_{0,m},L_p)$. Here, recall that the constant $L_p$ is the Lipschitz constant of $p$. Therefore by applying these estimates with \eqref{eq:B}, we get
\begin{align}
\begin{split}\label{estimate: finalA}
u_m(x,t) \leq &\mathbb E \Big[e^{-r(s-t)}u_m\big(X(s),s\big)\Big]+C\mathbb E\int_t^s|X(l)-x|\ud l\\
&+C(s-t)^2+\frac{s-t}{k}
\end{split}
\end{align}
for a constant $C:=C(m,\mu,p_{\text{max}},n,L_m,C_{0,m},L_p,r)$. By recalling \eqref{comp stoc dynamics} and utilizing Jensen's inequality and Ito's isometry, we see
\begin{align*}
\int_t^s\mathbb E|X(l)-x|\ud l\leq \tilde{C}\big((s-t)^2+(s-t)^{3/2}\big)
\end{align*}
for a constant $\tilde{C}:=\tilde{C}(m,\mu,p_{\text{max}},n)$. Thus, combining this with \eqref{estimate: finalA} and letting $s \to t_1$, we have
\begin{align}\label{eq: first iteration}
u_m(x,t) \leq &~\mathbb E \Big[e^{-r\Delta t}u_m\big(X(t_1),t_1\big)\Big]+C(\Delta t)^2+C(\Delta t)^{3/2}+\frac{\Delta t}{k}
\end{align}
for some generic constant $C$.

Next, we replicate the same argument as above in the time interval $E_2$. By Lemma \ref{lemma: phi cont}, it follows that there are a sequence $\mathcal C_2:=(a^{2,i},c_{2,i})_{i=1}^\infty$ and a covering $U_2:=\big(B(y^{2,i},r_{2,i})\big)_{i=1}^\infty$ of $\mathbb R^n$ such that
\begin{align}
\begin{split}\label{eq: need uniform cont}
\sup_{(b,d)\in \mathcal H_m}\bigg(&\Phi\big(a^{2,i},b,c_{2,i},d,Du_m(y,t_1),D^2u_m(y,t_1)\big)+ru_m(y,t_1)\bigg)\\
&\leq \partial_t u_m(y,t_1)+\frac{1}{k}
\end{split}
\end{align}
for all $y\in B(y^{2,i},r_{2,i})$. For $y\in \mathbb R^n$,  let $I_2(y)$ be the smallest index $i$ for which $y\in B(y^{2,i},r_{2,i})$ in the covering $\big(B(y^{2,i},r_{2,i})\big)_{i=1}^\infty$ of $\mathbb R^n$. Then, we define a function $z^2:\mathbb R^n \to \mathcal H_m$ by
$$
z^2(y)=\big(a^{2,I_2(y)},c_{2,I_2(y)}\big)
$$
for all $y\in \mathbb R^n$. Observe that we can construct $z^2$ in such a way that it is Borel measurable. Furthermore, we define a control $\big(a^2(l),c_2(l)\big)$ such that 
\begin{align*}
\big(a^2(l),c_2(l)\big)=\begin{cases}
(a^1,c_1),&~\text{ if }l\in E_1,\\
z^2\big(X(t_1)\big),&~\text{ if }l\in E_2.
\end{cases}
\end{align*}
By the inequality \eqref{eq: need uniform cont}, we can now repeat the argument above to get
\begin{align*}
u_m\big(X(t_1),t_1\big) \leq &~\mathbb E\Big[e^{-r\Delta t}u_m\big(X(t_2),t_2\big)\Big] +C(\Delta t)^2+C(\Delta t)^{3/2}+\frac{\Delta t}{k}.
\end{align*}
Thus, combining this estimate with \eqref{eq: first iteration}, it holds
\begin{align*}
u_m(x,t) \leq \mathbb E \Big[e^{-r2\Delta t}u_m\big(X(t_2),t_2\big)\Big]+2C(\Delta t)^2+2C(\Delta t)^{3/2}+\frac{2\Delta t}{k}.
\end{align*}

The idea is to replicate the argument in all time intervals $E_1, \dots E_k$. Indeed, after the $k$-th iteration, we get a control $\big(a^k(l),c_k(l)\big)$ such that
\begin{align*}
\big(a^k(l),c_k(l)\big)=\begin{cases}
(a^{k-1}(l),c_{k-1}(l)),&~\text{ if }l\in  \cup_{i=1}^{k-1} E_i \\
z^k\big(X(t_{k-1})\big),&~\text{ if }l\in  E_k.
\end{cases}
\end{align*}
Here, $z^k$ corresponds to the triplet $\big(\mathcal C_k,U_k,I_k(y)\big)$ in the same way as above. In particular, we have
\begin{align}\label{eq: final iteration}
u_m(x,t) \leq \mathbb E \Big[e^{-r(T-t)}g\big(X(T)\big)\Big]+(T-t)\big(C\Delta t+C(\Delta t)^{1/2}\big)+\Delta t,
\end{align}
because it holds $k=(T-t)/\triangle t$ and $u_m(z,T)=g(z)$ for all $z \in \mathbb R^n$.

Let $S\in \mathcal S_m$, and recall that the control $\big(b(l),d(l)\big)$ is arbitrary. We set
$$
\big(b(l),d(l)\big):=S\big(a^k(l),c_k(l)\big)
$$
for all $l\in [0,T]$. Then by \eqref{eq: final iteration}, it holds
\begin{align*}
&u_m(x,t) \leq \mathbb E \Big[e^{-r(T-t)}g\big(X(T)\big)\Big]+(T-t)\big(C\Delta t+C(\Delta t)^{1/2}\big)+\Delta t \\
& \leq \sup_{A\in \mathcal{AC}_m} \mathbb E\Big[e^{-r(T-t)}g\big(X(T)\big)\Big]+(T-t)\big(C\Delta t+C(\Delta t)^{1/2}\big)+\Delta t.
\end{align*}
Because $S\in \mathcal S_m$ is arbitrary, by letting $k \to \infty$, this yields
$$
u_m(x,t) \leq\inf_{S\in \mathcal{S}_m}\sup_{A\in \mathcal{AC}_m} \mathbb E\Big[e^{-r(T-t)}g\big(X(T)\big)\Big].
$$

The proof of the opposite inequality is analogous. Again, Lemma \ref{lemma: phi cont} implies that there are a sequence $\tilde{\mathcal C}_j:=(b^{j,i},d_{j,i})_{i=1}^\infty$ and a covering $\tilde{U}_j:=\big(B(\tilde{y}^{j,i},\tilde{r}_{j,i})\big)_{i=1}^\infty$ of $\mathbb R^n$ such that
\begin{align*}
\inf_{(a,c)\in \mathcal H_m}\bigg(&\Phi\big(a,b^{j,i},c,d_{j,i},Du_m(y,t_{j-1}),D^2u_m(y,t_{j-1})\big)+ru_m(y,t_{j-1})\bigg)\\
&\geq \partial_t u_m(y,t_{j-1})-\frac{1}{k}
\end{align*}
for all $y\in B(\tilde{y}^{j,i},\tilde{r}_{j,i})$ and $j\in\{2,\dots,k\}$, because $u_m$ is a solution to \eqref{eq: bdd 1}. Then by a similar reasoning to the above, we construct a control $\big(b^k(l),d_k(l)\big)$ to deduce
\begin{align}
\begin{split}\label{eq: final iteration opposite}
u_m(x,t) &\geq \mathbb E \Big[e^{-r(T-t)}g\big(X(T)\big)\Big]-C(T-t)\Delta t \\
&\hspace{1em}-C(T-t)(\Delta t)^{1/2}-\Delta t.
\end{split}
\end{align}
Let $A\in \mathcal{AC}_m$. We construct $S\in \mathcal S_m$ such that it holds
$$
S(A)=\big(b^k(l),d_k(l)\big)
$$
for all $l\in[0,T]$. Therefore, the inequality \eqref{eq: final iteration opposite} implies
\begin{align*}
u_m(x,t) &\geq \mathbb E \Big[e^{-r(T-t)}g\big(X(T)\big)\Big]-(T-t)\big(C\Delta t+C(\Delta t)^{1/2}\big)-\Delta t \\
& \geq  \inf_{S\in \mathcal{S}_m} \mathbb E\Big[e^{-r(T-t)}g\big(X(T)\big)\Big]-(T-t)\big(C\Delta t+C(\Delta t)^{1/2}\big)-\Delta t.
\end{align*}
Hence, by letting $k\to \infty$, we get
$$
u_m(x,t) \geq\inf_{S\in \mathcal{S}_m}\sup_{A\in \mathcal{AC}_m} \mathbb E\Big[e^{-r(T-t)}g\big(X(T)\big)\Big].
$$
Thus, the proof is complete.
\end{proof}
\end{lemma}

\section{Going to the limit: action sets without a uniform bound}\label{sec: limit}
In this section, we let bounds on the controls increase. To this end, we first show that viscosity solutions to the limiting equation are unique under suitable assumptions. Then by utilizing the stability principle and the equicontinuity of the families of viscosity solutions to the terminal value problems \eqref{eq: bdd 2} and \eqref{eq: bdd 1}, we see that there exist subsequences of solutions to  \eqref{eq: bdd 2} and \eqref{eq: bdd 1} converging uniformly to solutions of the limiting equation. The final part is to show that a subsequence of the corresponding value functions converges to a value function for the game without a uniform bound on the controls.

Let $J_0:=\mathbb R^n\times [0,T]\times \mathbb R \times \big(\mathbb R^n\setminus \{0\}\big) \times S(n)$, and define $F: J_0\to \mathbb R$ through
$$
F\big((x,t),\xi,\nu,M\big)=\big(p(x,t)-2\big)\frac{\langle M\nu,\nu\rangle}{|\nu|^2}+\tr(M)+\langle \mu,\nu \rangle-r\xi.
$$
Then, the limiting terminal value problem for \eqref{eq: bdd 2} and \eqref{eq: bdd 1} as $m\to \infty$ is
\begin{align}
\begin{split}\label{eq: bdd 3}
\begin{cases}
\partial_tu+F\big((x,t),u,Du,D^2u\big)=0&\text{in }\mathbb R^n\times (0,T), \\
u(x,T)=g(x)&\text{on }\mathbb R^n.
\end{cases}
\end{split}
\end{align}
As before, this equation is understood in the viscosity sense. We take care of the points, where the gradient of the underlying function in the operator $F$ vanishes, via semicontinuous envelopes. Let us denote
$$
F_*\big((x,t),\xi,\nu,M\big):=\liminf_{\tilde{\nu} \to \nu}F\big((x,t),\xi,\tilde{\nu},M\big)
$$
for all $(x,t)\in \mathbb R^n \times [0,T]$, $\xi\in \mathbb R$, $\nu \in \mathbb R^n$ and $M\in S(n)$, and $F^*:=-(-F)_*$. The following definition parallels Definition \ref{def: belmann-i}.
\begin{definition}\label{limit def visc}
($i$) A lower semicontinuous function $\ol{u}:\mathbb R^n \times [0,T] \to \mathbb R$ is a viscosity supersolution to \eqref{eq: bdd 3}, if it satisfies the growth bound \eqref{linear growth}, 
$$
\ol{u}(x,T)\geq g(x)
$$
for all $x \in \mathbb R^n$, and if the following holds. For all $(x_0,t_0)\in \mathbb R^n \times (0,T)$ and for all $\phi \in C^{2,1}\big(\mathbb R^n \times (0,T)\big)$ such that
\begin{itemize}
\item $\ol{u}(x_0,t_0)=\phi(x_0,t_0)$
\item $\ol{u}(x,t)>\phi(x,t)$ for all  $(x,t)\not=(x_0,t_0)$
\end{itemize}
it holds
$$
\partial_t\phi(x_0,t_0)+F\big((x_0,t_0),\phi(x_0,t_0),D\phi(x_0,t_0),D^2\phi(x_0,t_0)\big)\leq 0
$$
whenever $D\phi(x_0,t_0)\not=0$, and 
$$
\partial_t\phi(x_0,t_0)+F_*\big((x_0,t_0),\phi(x_0,t_0),0,D^2\phi(x_0,t_0)\big)\leq 0,
$$
whenever $D\phi(x_0,t_0)=0$.

($ii$) An upper semicontinuous function $\underline{u}:\mathbb R^n \times [0,T] \to \mathbb R$ is a viscosity subsolution to \eqref{eq: bdd 3}, if it satisfies the growth bound \eqref{linear growth}, 
$$
\underline{u}(x,T)\leq g(x)
$$
for all $x \in \mathbb R^n$, and if the following holds. For all $(x_0,t_0)\in \mathbb R^n \times (0,T)$ and for all $\phi \in C^{2,1}\big(\mathbb R^n \times (0,T)\big)$ such that
\begin{itemize}
\item $\underline{u}(x_0,t_0)=\phi(x_0,t_0)$
\item $\underline{u}(x,t)<\phi(x,t)$ for all  $(x,t)\not=(x_0,t_0)$
\end{itemize}
it holds
$$
\partial_t\phi(x_0,t_0)+F\big((x_0,t_0),\phi(x_0,t_0),D\phi(x_0,t_0),D^2\phi(x_0,t_0)\big)\geq 0,
$$
whenever $D\phi(x_0,t_0)\not=0$, and 
$$
\partial_t\phi(x_0,t_0)+F^*\big((x_0,t_0),\phi(x_0,t_0),0,D^2\phi(x_0,t_0)\big)\geq 0,
$$
whenever $D\phi(x_0,t_0)=0$.

($iii$) If a function $u :\mathbb R^n \times [0,T] \to \mathbb R$ is a viscosity supersolution and a subsolution to \eqref{eq: bdd 3}, then $u$ is a viscosity solution to \eqref{eq: bdd 3}.
\end{definition}

\begin{remark}\label{equiv def}
Observe that for any test function $\phi \in C^{2,1}\big(\mathbb R^n \times (0,T)\big)$ such that $D\phi(x_0,t_0)\not=0$ or $D^2\phi(x_0,t_0)=\textbf 0$ in the Definition \ref{limit def visc}, it holds
\begin{align*}
&F_*\big((x_0,t_0),\phi(x_0,t_0),D\phi(x_0,t_0),D^2\phi(x_0,t_0)\big) \\
&=F^*\big((x_0,t_0),\phi(x_0,t_0),D\phi(x_0,t_0),D^2\phi(x_0,t_0)\big) 
\end{align*}
for all $(x_0,t_0)\in \mathbb R^n \times (0,T)$. 
\end{remark}
To prove a comparison principle for the equation \eqref{eq: bdd 3}, we follow the path developed in \cite{gigagis91}, see also \cite{chengg91,juutinenlm01,kawohlmp12}. Here, the main difficulties arise from the $(x,t)$ dependence in $F$ as well as from the unboundedness of the domain.
\begin{theorem}\label{comparison limit eq}
Let $\underline{u}$ and $\ol{u}$ be continuous viscosity sub- and supersolutions to \eqref{eq: bdd 3} in the sense of Definition \ref{limit def visc}, respectively. Then, it holds
$$
\underline{u}(x,t)\leq \ol{u}(x,t)
$$
for all $(x,t)\in \mathbb R^n\times [0,T]$.
\begin{proof}
The proof is by contradiction. We assume that 
\begin{align}
\begin{split}\label{cp comparison}
\alpha&:=\sup_{\mathbb R^n \times [0,T]}\big(\underline{u}-\ol{u}\big)>0.
\end{split}
\end{align}
Let $\eps,\delta,\gamma>0$, and define
$$
w_{\eps,\delta,\gamma}(x,y,t)=\underline{u}(x,t)-\ol{u}(y,t)-\frac{1}{4\eps}|x-y|^4-B_{\delta,\gamma}(x,y,t)
$$
for all $x,y\in \mathbb R^n$ and $t\in(0,T]$, where
\begin{align}\label{infty barrier}
B_{\delta,\gamma}(x,y,t):=\delta(|x|^2+|y|^2)+\gamma t^{-1}.
\end{align}
The function $B_{\delta,\gamma}$ plays the role of a barrier for space infinity and $t=0$. 

We can show, see \cite[Proposition 2.3]{gigagis91}, that there are constants $K,K'>0$ independent of $x,y,t$ such that
\begin{align}\label{eq: rough growth estimate}
\underline{u}(x,t)-\ol{u}(y,t)\leq K|x-y|+K'(1+t)
\end{align}
for all $x,y\in \mathbb R^n$ and $t\in[0,T]$. Indeed, because for $R'>0$ it holds
$$
\Big|F\big((x,t),\xi,p,M\big)\Big|\leq (p_{\text{max}}-2+n+|\mu|)R'+r|\xi|<\infty
$$
for all $(x,t,\xi,p,M)\in J_0$ such that $|p|\leq R'$ and $||M||\leq R'$, we can utilize the same arguments as in \cite[Proposition 2.3]{gigagis91}. Therefore by the estimate \eqref{eq: rough growth estimate}, it holds $\alpha<\infty$ in \eqref{cp comparison}. 

We denote by $(\hat{x},\hat{y},\hat{t})$ a maximum point of $w_{\eps,\delta,\gamma}$ in $\mathbb R^n \times \mathbb R^n \times [0,T]$. The growth condition \eqref{linear growth} and the barrier \eqref{infty barrier} ensure that $w_{\eps,\delta,\gamma}(x,y,t)<0$, when $x,y$ are outside a compact set $E\subset \mathbb R^n \times \mathbb R^n$ depending on $\delta$, and $t\in(0,T]$. Therefore, because $w_{\eps,\delta,\gamma}$ is continuous and \eqref{cp comparison} holds with $\alpha<\infty$, the maximum point exists for all $\delta,\gamma$ small enough and any $\eps$. Furthermore by \eqref{cp comparison}, we can find $(x_0,t_0)\in \mathbb R^n \times [0,T]$ such that 
$$
\underline{u}(x_0,t_0)-\ol{u}(x_0,t_0)>\alpha-\eps/3.
$$ 
Because $\underline{u}-\ol{u}$ is continuous, we may assume that $t_0>0$. Consequently, for $\eps<\alpha$ there are $\delta_0:=\delta_{0}(\eps)>0$ and $\gamma_0:=\gamma_{0}(\eps)>0$ such that
\begin{equation}\label{w bounded uniformlu from below}
w_{\eps,\delta,\gamma}(\hat{x},\hat{y},\hat{t})\geq \underline{u}(x_0,t_0)-\ol{u}(x_0,t_0)-2\delta|x_0|-\gamma t_0^{-1}>\alpha-\eps
\end{equation}
for all $\delta<\delta_0$ and $\gamma<\gamma_0$. Let $\eps<\alpha/2,\delta<\delta_0$ and $\gamma<\gamma_0$. Then by \eqref{w bounded uniformlu from below} we can estimate
\begin{align*}
\underline{u}(\hat{x},\hat{t})-\ol{u}(\hat{y},\hat{t})&>\frac{1}{4\eps}|\hat{x}-\hat{y}|^4+B_{\delta,\gamma}(\hat{x},\hat{y},\hat{t})\geq \frac{1}{4\eps}|\hat{x}-\hat{y}|^4.
\end{align*}
This and \eqref{eq: rough growth estimate} imply
$$
|\hat{x}-\hat{y}| \leq   4\eps\big(K|\hat{x}-\hat{y}|^{-3}+K'(1+T)|\hat{x}-\hat{y}|^{-4}\big).
$$
Therefore, we have $|\hat{x}-\hat{y}|<C$ for some $C<\infty$ independent of $\eps,\delta$ and $\gamma$. Moreover, it holds
\begin{equation}\label{maximum point convergence rate}
|\hat{x}-\hat{y}| \leq \max\big\{\eps^{1/8},4K\eps^{5/8}+4K'(1+T)\sqrt{\eps}\big\}=:\zeta(\eps).
\end{equation}
By an analogous argument, we can deduce $\hat{t}>0$. Because it holds $\underline{u}(z,T)\leq \ol{u}(z,T)$ for all $z\in \mathbb R^n$ by the assumptions, the inequality \eqref{w bounded uniformlu from below} yields $\hat{t}<T$. In addition, because $|\hat{x}-\hat{y}|$ is bounded, the estimate \eqref{eq: rough growth estimate} implies that $w_{\eps,\delta,\gamma}(\hat{x},\hat{y},\hat{t})$ is uniformly bounded from above with respect to $\delta$. Hence, because $w_{\eps,\delta,\gamma}(\hat{x},\hat{y},\hat{t})$ increases as $\delta \to 0$, the quantity $\lim_{\delta \to 0}w_{\eps,\delta,\gamma}(\hat{x},\hat{y},\hat{t})$ exists. Therefore by denoting $(\tilde{x},\tilde{y},\tilde{t})$ a global maximum point of  $w_{\eps,\delta/2,\gamma}$, we have 
$$
w_{\eps,\delta/2,\gamma}(\tilde{x},\tilde{y},\tilde{t})\geq w_{\eps,\delta,\gamma}(\hat{x},\hat{y},\hat{t})+\delta/2\big(|\hat{x}|^2+|\hat{y}|^2\big)
$$
implying
\begin{equation}\label{delta convergence}
\delta \big(|\hat{x}|^2+|\hat{y}|^2\big) \to 0
\end{equation}
as $\delta \to 0$.

By theorem of sums, see \cite[Theorem 8.3]{grandallil92}, there exist symmetric matrices $X:=X(\eps,\delta)$ and $Y:=Y(\eps,\delta)$, and real numbers $\tau_{\underline{u}}$ and $\tau_{\ol{u}}$, such that $\tau_{\underline{u}}-\tau_{\ol{u}}=\partial_t B_{\delta,\gamma}(\hat{x},\hat{y},\hat{t})=-\gamma \hat{t}^{-2}$ and
\begin{align}
\begin{split}\label{super jet u2}
\Big(\tau_{\underline{u}},\eps^{-1}\abs{\hat{x}-\hat{y}}^2(\hat{x}-\hat{y})+2\delta\hat{x},\,X\Big)&\in \ol{\mathcal P}^{2,+}\underline{u}(\hat{x},\hat{t}),\\
\Big(\tau_{\ol{u}},\eps^{-1}\abs{\hat{x}-\hat{y}}^2(\hat{x}-\hat{y})-2\delta\hat{y},\,Y\Big)&\in  \ol{\mathcal P}^{2,-}\ol{u}(\hat{y},\hat{t}).
\end{split}
\end{align}
Furthermore by computing the second derivatives of the function $B_{\delta,\gamma}(x,y,t)+\frac{1}{4\eps}|x-y|^4$, it holds 
\begin{align}
\begin{split}\label{thm of sum for hessians}
\left[ \begin{array}{cc}
X&0\\
0&-Y
\end{array} \right]
\leq 
&~(1+4\eps\delta)\left[ \begin{array}{cc}
M&-M\\
-M&M
\end{array} \right]
+2\eps\left[ \begin{array}{cc}
M^2&-M^2\\
-M^2&M^2
\end{array} \right] \\
&+2\delta(1+2\delta)\left[ \begin{array}{cc}
 I&0\\
0& I
\end{array} \right]
\end{split}
\end{align}
with
$$
M:=\eps^{-1}\Big( 2(\hat{x}-\hat{y})\otimes (\hat{x}-\hat{y})+\abs{\hat{x}-\hat{y}}^2I\Big),
$$
and
\begin{align}\label{lower bound hessian}
\left[ \begin{array}{cc}
X&0\\
0&-Y
\end{array} \right]\geq-(\eps^{-1}+3\eps^{-1}|\hat{x}-\hat{y}|^2+2\delta)\left[ \begin{array}{cc}
I&0\\
0&I
\end{array} \right].
\end{align}
Thus, because $\underline{u}$ is a subsolution and $\ol{u}$ is a supersolution, it holds by \eqref{super jet u2}
\begin{align}
\begin{split}\label{vali estimaatti1}
\tau_{\underline{u}}+F^*\big((\hat{x},\hat{t}),\underline{u}(\hat{x},\hat{t}),\eps^{-1}\abs{\hat{x}-\hat{y}}^2(\hat{x}-\hat{y})+2\delta \hat{x},X\big) &\geq 0, \\
\tau_{\ol{u}}+F_*\big((\hat{y},\hat{t}),\ol{u}(\hat{y},\hat{t}),\eps^{-1}\abs{\hat{x}-\hat{y}}^2(\hat{x}-\hat{y})-2\delta \hat{y},Y\big)&\leq 0,
\end{split}
\end{align}
see also Remark \ref{equiv def}.

We consider two different cases depending on the behavior of $\hat x - \hat y$ as $\delta \to 0$. First, assume that $\hat{x}-\hat{y} \to 0$ as $\delta \to 0$. Then by the estimate \eqref{thm of sum for hessians}, it holds
$$
\limsup_{\delta \to 0}\langle Xz,z\rangle\leq 0 ~\text{and } \liminf_{\delta \to 0}\langle Yz,z\rangle \geq 0
$$
for all $z\in \mathbb R^n$.  Thus by combining this with \eqref{vali estimaatti1}, and recalling \eqref{w bounded uniformlu from below}, the degenerate ellipticity of $F$ and $\delta \hat x, \delta \hat y\to 0$ as $\delta \to 0$ by \eqref{delta convergence}, we can estimate
\begin{align*}
\gamma T^{-2}&\leq  \limsup_{\delta \to 0}F^*\big((\hat{x},\hat{t}),\underline{u}(\hat{x},\hat{t}),0,\textbf{0}\big)- \liminf_{\delta \to 0}F_*\big((\hat{y},\hat{t}),\ol{u}(\hat{y},\hat{t}),0,\textbf{0}\big) \\
&\leq 0.
\end{align*}
Hence, because it holds $\gamma>0$, we have found a contradiction.

Next, we assume $\hat{x}-\hat{y}\to \eta\not=0$ for some subsequence still denoted by $(\delta)$. For brevity, let us denote
\begin{align*}
\tilde{\xi}_x&:=\eps^{-1}\abs{\hat{x}-\hat{y}}^2(\hat{x}-\hat{y})+2\delta\hat{x},\\
\tilde{\xi}_y&:=\eps^{-1}\abs{\hat{x}-\hat{y}}^2(\hat{x}-\hat{y})-2\delta\hat{y},
\end{align*}
$\xi_x:=\tilde{\xi}_x/|\tilde{\xi}_x|$ and $\xi_y:=\tilde{\xi}_y/|\tilde{\xi}_y|$ assuming $\tilde{\xi}_x,\tilde{\xi}_y\not=0$. Then, because of \eqref{w bounded uniformlu from below} and \eqref{vali estimaatti1}, we can estimate 
\begin{align}
\begin{split}\label{vali estimaatti}
0<&~\big(p(\hat{x},\hat{t})-2\big)\big\langle X\xi_x,\xi_x\big \rangle-\big(p(\hat{y},\hat{t})-2\big)\big\langle Y\xi_y,\xi_y\big \rangle \\
&+\sum_{i=1}^n\lambda_i\big(X-Y\big)+2\langle \mu,\delta \hat{x}+\delta \hat{y}\rangle -r\alpha/2,
\end{split}
\end{align}
where $\lambda_i$ denotes the $i$-th eigenvalue of the corresponding matrix. Because the first two matrices in the right-hand side of \eqref{thm of sum for hessians} annihilate, we have
\begin{equation}\label{annihilate}
X -Y \leq 4\delta(1+2\delta)I.
\end{equation}
Thus to complete the proof, we need to estimate the first two terms in the right-hand side of \eqref{vali estimaatti}. 

Let us define $\xi_\delta:=(\hat{x}-\hat{y})/|\hat{x}-\hat{y}|\in \mathbb S^{n-1}$ for all $\delta$ small enough. Then, it holds
\begin{equation}\label{point convergence for xi1}
\xi_{\delta}\to \eta/|\eta|
\end{equation}
as $\delta \to 0$. Observe that by the convergence \eqref{delta convergence}, it also holds
\begin{equation}\label{point convergence for xi2}
\xi_x,\xi_y\to\eta/|\eta|
\end{equation}
as  $\delta \to 0$. Furthermore by \eqref{thm of sum for hessians} and \eqref{lower bound hessian}, $X$ and $Y$ are uniformly bounded with respect to $\delta$, see also \cite[Lemma 5.3]{ishii89}. Thus, because the function $p$ is bounded, the convergences \eqref{point convergence for xi1} and \eqref{point convergence for xi2} imply
\begin{align}
\begin{split}\label{where to use key ineq}
&\big(p(\hat{x},\hat{t})-2\big)\big\langle X\xi_x,\xi_x\big \rangle-\big(p(\hat{y},\hat{t})-2\big)\big\langle Y\xi_y,\xi_y\big \rangle \\
&=\big(p(\hat{x},\hat{t})-1\big)\big\langle X\xi_\delta,\xi_\delta\big \rangle-\big(p(\hat{y},\hat{t})-1\big)\big\langle Y\xi_\delta,\xi_\delta\big \rangle \\
&\hspace{1em}-\big \langle (X-Y)\xi_\delta,\xi_\delta\big \rangle+E_\delta(\hat x,\hat y, \hat t)
\end{split}
\end{align}
for some error $E_\delta(\hat x,\hat y, \hat t)$ such that
$$
E_\delta(\hat x,\hat y, \hat t) \to 0
$$ 
as $\delta \to 0$. For the vector 
$$
\big(\xi_\delta^T\sqrt{p(\hat{x},\hat{t})-1},\xi_\delta^T\sqrt{p(\hat{y},\hat{t})-1}\big)\in \mathbb R^{2n}
$$
in the estimate \eqref{thm of sum for hessians}, it holds
\begin{align}
\begin{split}\label{final estimate to get contra}
& (p(\hat{x},\hat{t})-1)\langle X\xi_\delta,\xi_\delta\big \rangle- (p(\hat{y},\hat{t})-1)\langle Y\xi_\delta,\xi_\delta\big \rangle \\
&\leq \Big(\sqrt{p(\hat{x},\hat{t})-1}-\sqrt{p(\hat{y},\hat{t})-1}\Big)^2\Big((1+4\eps\delta)\big\langle M\xi_\delta,\xi_\delta\big \rangle \\
&\hspace{8em}+2\eps\big\langle M^2\xi_\delta,\xi_\delta\big \rangle\Big) +4(p_{\text{max}}-1)\delta(1+2\delta) \\
&\leq \frac{L_p^2}{4(p_{\text{min}}-1)}|\hat{x}-\hat{y}|^2\Big((1+4\eps\delta)3\eps^{-1}|\hat{x}-\hat{y}|^2+18\eps^{-1}|\hat{x}-\hat{y}|^4\Big) \\
&\hspace{1em}+4(p_{\text{max}}-1)\delta(1+2\delta),
\end{split}
\end{align}
where $L_p$ is the Lipschitz constant of $p$. Moreover by the estimates \eqref{w bounded uniformlu from below} and \eqref{maximum point convergence rate}, it holds
\begin{align*}
\frac{|\hat{x}-\hat{y}|^4}{4\eps}&<\underline{u}(\hat{x},\hat{t})-\ol{u}(\hat{y},\hat{t})-\alpha+\eps \\
&\leq \sup_{|x-y|<\zeta(\eps),t\in [0,T]}\big(\underline{u}(x,t)-\ol{u}(y,t)\big)-\alpha+\eps.
\end{align*}
This estimate, together with \eqref{cp comparison}, implies
\begin{align*}
\lim_{\eps \to 0}\limsup_{\delta,\gamma \to 0}\frac{|\hat{x}-\hat{y}|^4}{\eps}=0.
\end{align*}
Therefore by combining this, \eqref{delta convergence}, \eqref{annihilate}, \eqref{where to use key ineq} and \eqref{final estimate to get contra} with the estimate \eqref{vali estimaatti}, we have found a contradiction by first letting $\delta,\gamma \to 0$ and then $\eps \to 0$. Hence, the proof is complete.
\end{proof}
\end{theorem}

A typical phenomenon for equations of $p$-Laplacian type is that the set of test functions used in their definition can be reduced.
\begin{lemma}\label{test f reduction}
Let $u:\mathbb R^n\times [0,T] \to \mathbb R$ be continuous. Then, to test whether or not $u$ is a viscosity super- or subsolution at $(x_0,t_0)$ in the sense of Definition \ref{limit def visc}, it is enough to consider test functions $\phi\in C^{2,1}\big(\mathbb R^n \times (0,T)\big)$ such that either
\begin{itemize}
\item $D\phi(x_0,t_0)\not = 0$\text{ or}
\item $D\phi(x_0,t_0)=0$\text{ and } $D^2\phi(x_0,t_0)=0$.
\end{itemize}
\begin{proof}
We only provide the proof in the context of supersolutions. Let $(x_0,t_0)\in \mathbb R^n\times(0,T)$. Assume that there exist $\delta>0$ and a test function $\phi\in C^{2,1}\big(\mathbb R^n \times (0,T)\big)$ such that $u(x_0,t_0)=\phi(x_0,t_0)$, $u(x,t)>\phi(x,t)$ for $(x,t)\not=(x_0,t_0)$, $D \phi(x_0,t_0)=0$, $D^2\phi(x_0,t_0)\not=\textbf{0}$ and
\begin{equation}\label{failed test function}
0<\partial_t\phi(x_0,t_0)+F_*\big((x_0,t_0),\phi(x_0,t_0),0,D^2\phi(x_0,t_0)\big)-\delta.
\end{equation}
Observe that $u-\phi$ has a strict global minimum at $(x_0,t_0)$. We define a function
$$
w_j(x,t,y,s):=u(x,t)-\phi(y,s)+\frac{j}{4}|x-y|^4+\frac{j}{2}(t-s)^2
$$
for $x,y\in \mathbb R^n, t,s\in [0,T]$. Let $R:=\max\{2|x_0|,1\}>0$, and denote by $(x_j,t_j,y_j,s_j)$ a minimum point of $w_j$ on a compact set $K:=\ol{B}_R(0) \times [0,T]\times \ol{B}_R(0) \times [0,T]$. Because $w_j(x_j,t_j,y_j,s_j)$ increases as $j$ increases, and it is bounded from above by $w_j(x_0,t_0,x_0,t_0)=0$ for all $j$, the limit $$\lim_{j\to\infty}w_j(x_j,t_j,y_j,s_j)<\infty$$ exists. Consequently, the estimate
$$
w_{j/2}(x_{j/2},t_{j/2},y_{j/2},s_{j/2})\leq w_j(x_j,t_j,y_j,s_j)-\frac{j}{8}|x_j-y_j|^4-\frac{j}{4}(t_j-s_j)^2
$$
implies
\begin{align}\label{upper estimate}
j|x_j-y_j|^4+j(t_j-s_j)^2 \to 0
\end{align}
as $j \to \infty$. Furthermore, because the global minimum of $u-\phi$ is strict, it holds 
\begin{align}\label{points converge}
(x_j,t_j,y_j,s_j)\to (x_0,t_0,x_0,t_0)
\end{align}
as $j\to \infty$. In particular, the point $(x_j,t_j,y_j,s_j)$ is not on the boundary of the set $K$ for all $j$ large enough, because it holds $(x_0,t_0)\in B_R(0)\times(0,T)$.

We prove the case $x_j=y_j$ for an infinite sequence of $j$:s, and consider only such indices $j$. The proof in the case $x_j\not=y_j$ for all $j$ large enough is similar to the proof of Theorem \ref{comparison limit eq}, see also \cite{chengg91,juutinenlm01}. By denoting $\varphi(x,y):=\frac j4|x-y|^4$, it holds
$$
D_{x}\varphi(x_j,y_j)=-D_{y}\varphi(x_j,y_j)=0~\text{and }D_{xx}^2\varphi(x_j,y_j)=D_{yy}^2\varphi(x_j,y_j)=\textbf{0}.
$$
Furthermore, the function 
$$
(y,s)\mapsto \phi(y,s)-\varphi(x_j,y)-\frac j2(t_j-s)^2
$$
has a local maximum at $(y_j,s_j)$. These imply $D\phi(y_j,s_j)=-D_{y}\varphi(x_j,y_j)=0$, $\partial_t \phi(y_j,s_j)=-j(t_j-s_j)$ and $D^2\phi(y_j,s_j)\leq -D_{yy}^2\varphi(x_j,y_j)=\textbf{0}$. Thus, because $p$ and $(y,s)\mapsto \lambda_i\big(D^2\phi(y,s)\big)$ for any $i$ are continuous with $\lambda_i$ denoting the $i$-th eigenvalue of the corresponding matrix, the assumption \eqref{failed test function} and the convergence \eqref{points converge} yield
\begin{align}
\begin{split}\label{first est to get contradiction}
0&<\partial_t\phi(y_j,s_j)+\lambda_{\text{max}}\Big(\big(p(y_j,s_j)-1\big)D^2\phi(y_j,s_j)\Big)\\
&\hspace{1em}+\sum_{i\not=i_{\text{min}}}\lambda_i\Big(D^2\phi(y_j,s_j)\Big)-r\phi(y_j,s_j)-\frac \delta2 \\
&\leq-j(t_j-s_j)-r\phi(y_j,s_j)-\frac\delta2
\end{split}
\end{align}
for all $j$ large enough. Furthermore, because the function
\begin{align*}
(x,t)\mapsto \Psi(x,t):=&-\varphi(x,y_j)-\frac j2(t-s_j)^2+\varphi(x_j,y_j)+\frac j2(t_j-s_j)^2\\
&+u(x_j,t_j)
\end{align*}
tests $u$ from below at $(x_j,t_j)$, and it holds $D_x\Psi(x_j,t_j)=0$, we have
$$
0\geq \Psi_t(x_j,t_j)+F_*\big((x_j,t_j),u(x_j,t_j),0,D^2_{xx}\Psi(x_j,t_j)\big).
$$
Thus, because it holds $\Psi_t(x_j,t_j)=-j(t-s_j)$ and $D^2_{xx}\Psi(x_j,t_j)=\bold{0}$, by combining this and \eqref{first est to get contradiction}, we get
$$
0<r\big(u(x_j,t_j)-\phi(y_j,s_j)\big)-\delta/2.
$$
Hence, because $u$ is continuous and \eqref{points converge} holds, we find a contradiction for all $j$ large enough.
\end{proof}
\end{lemma}

The following lemma suggests that $F$ is the correct limiting equation in our setting. The proof for the equation $F_m^+$ is analogous.
\begin{lemma}\label{lemma: limit eq}
Let $(x_m,t_m),(x,t)\in \mathbb R^n\times [0,\infty)$, $\xi_m,\xi \in \mathbb R$, $\nu_m,\nu\in \mathbb R^n\setminus \{0\}$ and $M_m,M\in S(n)$ be such that 
$$
(x_m,t_m)\to(x,t),~\xi_m \to \xi, ~\nu_m\to \nu~\text{and}~M_m\to M
$$
as $m\to \infty$. Then, it holds
$$
F_m^-\big((x_m,t_m),\xi_m,\nu_m,M_m\big)\to -F\big((x,t),\xi,\nu,M\big)
$$
as $m\to \infty$.
\begin{proof}
It is clear that $\langle \mu,\nu_m\rangle \to \langle \mu,\nu\rangle$ and $r\xi_m\to r\xi$ as $m\to \infty$. To complete the proof, we utilize the key inequality
\begin{equation}\label{key ineq}
\big\langle\nu_m/|\nu_m|+\xi,\nu_m\big\rangle\geq 0
\end{equation}
whenever $\xi\in\mathbb S^{n-1}$. 

We set $$\tilde{\Phi}_m:=\inf_{(a,c)\in \mathcal H_m} \sup_{(b,d)\in \mathcal H_m}\Big[-\tr\Big(\mathcal A_{a,b}^{(x_m,t_m)}M_m\Big)-(c+d)\langle a+b,\nu_m\rangle\Big].$$ Because $\big(\nu_m/|\nu_m|,0\big)\in \mathcal H_m$, it holds
\begin{align*}
\tilde{\Phi}_m\leq& \sup_{(b,d)\in \mathcal H_m}\Big[-\tr\Big(\mathcal A_{\frac{\nu_m}{|\nu_m|},b}^{(x_m,t_m)}M_m\Big)-d\langle \nu_m/|\nu_m|+b,\nu_m\rangle\Big].
\end{align*}
Therefore, this estimate and \eqref{key ineq} imply
$$
\tilde{\Phi}_m \leq n\Lambda\big|\big|M_m\big|\big|,
$$
where $\Lambda$ is defined in \eqref{iso lambda}. Hence, $\tilde{\Phi}_m$ is bounded from above as $m\to \infty$.

Because  $\big(-\nu_m/|\nu_m|,m\big)\in\mathcal H_m$, we can estimate 
\begin{align*}
\tilde{\Phi}_m \geq& \inf_{(a,c)\in \mathcal H_m} \Big[-\tr\Big(\mathcal A_{a,-\frac{\nu_m}{|\nu_m|}}^{(x_m,t_m)}M_m\Big)-(c+m)\langle a -\nu_m/|\nu_m|,\nu_m\rangle\Big].
\end{align*}
Now, \eqref{key ineq} implies that the second term after the infimum is bounded from below as $m\to \infty$. Hence by the definition of the infimum, there exists $(a_m,c_m)\in \mathcal H_m$ such that 
\begin{align}
\begin{split}\label{eq: infimum choice}
\tilde{\Phi}_m &\geq-\tr\Big(\mathcal A_{a_m,-\frac{\nu_m}{|\nu_m|}}^{(x_m,t_m)}M_m\Big)\\
&\hspace{1em}-(c_m+m)\langle a_m-\nu_m/|\nu_m|,\nu_m\rangle -\frac{1}{m}.
\end{split}
\end{align}

Next, we prove that 
\begin{equation}\label{eq: inf limit}
a_m\to\frac{\nu}{|\nu|}
\end{equation}
as $m\to \infty$. To establish this, it suffices to show that for given $\eta>0$, there is $m_0:=m_0(\eta)$ such that
$$
\langle a_m,\nu_m\rangle \geq |\nu_m|-\eta
$$
for all $m\geq m_0$. We assume, on the contrary, that there is $\eta>0$ such that for all $m\geq 0$
$$
\langle a_m,\nu_m\rangle<|\nu_m|-\eta.
$$
Thus in this case, \eqref{eq: infimum choice} implies
\begin{align*}
\tilde{\Phi}_m &\geq  -n\Lambda\big|\big|M_m\big|\big|+\eta(c_m+m)-\frac{1}{m}.
\end{align*}
This contradicts the boundedness of $\tilde{\Phi}_m$ as $m\to \infty$, and hence, \eqref{eq: inf limit} holds.

Recall that the function $p$ is continuous which implies $p(x_m,t_m)\to p(x,t)$ as $m\to \infty$. Therefore by combining the assumptions, \eqref{key ineq} and \eqref{eq: inf limit} with \eqref{eq: infimum choice}, we get
\begin{align*}
\liminf_{m \to \infty}\tilde{\Phi}_m &\geq -\tr\Big(\mathcal A_{\frac{\nu}{|\nu|},-\frac{\nu}{|\nu|}}^{(x,t)}M\Big) \\
&= -\big(p(x,t)-2\big)\frac{\langle M\nu,\nu\rangle}{|\nu|^2}-\tr(M).
\end{align*}
Thus, we have proven
$$
\liminf_{m\to \infty}F_m^-\big((x_m,t_m),\xi_m,\nu_m,M_m\big)\geq -F\big((x,t),\xi,\nu,M\big).
$$

Next, we prove that
\begin{equation}\label{eq: A2}
\limsup_{m\to \infty}F_m^-\big((x_m,t_m),\xi_m,\nu_m,M_m\big)\leq -F\big((x,t),\xi,\nu,M\big).
\end{equation}
Again, as $\big(\nu_m/|\nu_m|,m\big)\in \mathcal H_m$, we have
\begin{align*}
\tilde{\Phi}_m &\leq \sup_{(b,d)\in \mathcal H_m}\Big[-\tr\Big(\mathcal A_{\frac{\nu_m}{|\nu_m|},b}^{(x_m,t_m)}M_m\Big)-(m+d)\langle \nu_m/|\nu_m|+b,\nu_m\rangle\Big].
\end{align*}
Because the second term after the supremum is bounded from above by \eqref{key ineq}, we find $\big(b_m,d_m\big)\in\mathcal H_m$ such that
\begin{align}\label{calc: A1}
\tilde{\Phi}_m \leq &-\tr\Big(\mathcal A_{\frac{\nu_m}{|\nu_m|},b_m}^{(x_m,t_m)}M_m\Big)-(m+d_m)\langle \nu_m/|\nu_m|+b_m,\nu_m\rangle+\frac{1}{m}
\end{align}
by the definition of the supremum. Moreover, $\tilde{\Phi}_m$ is bounded also from below, because we can use \eqref{key ineq} and estimate the supremum in $\tilde{\Phi}_m$ with the choice $(-\nu_m/|\nu_m|,0)\in \mathcal H_m$. This and the estimate \eqref{calc: A1} imply $b_m\to-\nu/|\nu|$ as $m\to \infty$ in a similar way to the above. Therefore, this, together with the estimate \eqref{key ineq} in the inequality \eqref{calc: A1}, by taking $\limsup_{m \to \infty}$, completes the proof of \eqref{eq: A2}. 
\end{proof}
\end{lemma}

For all $M\in S(n)$, we utilize the Pucci operators
$$
P^+(M):=\sup_{A\in \mathcal A_{\lambda,\Lambda}}\tr(AM)
$$
and
$$
P^-(M):=\inf_{A\in \mathcal A_{\lambda,\Lambda}}\tr(AM),
$$
where $\mathcal A_{\lambda,\Lambda}\subset S(n)$ is the set of symmetric $n\times n$ matrices whose eigenvalues belong to $[\lambda,\Lambda]$.

\begin{lemma}\label{lemma: equicontinuity}
Let $u_m$ be the unique solution to \eqref{eq: bdd 2} ensured by Proposition \ref{prop: unique sol}. Then, the function $u_m$ is H\"{o}lder continuous on $\mathbb R^n \times [0,T]$ with a H\"{o}lder constant independent of $m$. In particular, the sequence
$$
\{u_m:m \geq 1\}
$$
is equicontinuous on $\mathbb R^n \times [0,T]$.

\begin{proof}
Let $m\geq 1$ and $(x,t)\in \mathbb R^n \times (0,T)$. Furthermore, let $\varphi \in C^2\big(\mathbb R^n\times (0,T)\big)$ test $u_m$ from below at $(x,t)$. First, we assume $D\varphi(x,t)\not=0$. Because $u_m$ is a supersolution to \eqref{eq: bdd 2}, we can find a vector $b_m$ on a compact set $\mathbb S^{n-1}$ such that 
\begin{align*}
0&\geq \partial_t \varphi(x,t)+\tr\Big(\mathcal A_{\frac{D\varphi(x,t)}{|D\varphi(x,t)|},b_m}^{(x,t)}D^2\varphi(x,t)\Big)+\big\langle \mu,D\varphi(x,t)\big\rangle-r\varphi(x,t) \\
&\geq \partial_t \varphi(x,t)+P^-(D^2\varphi(x,t))+\big\langle \mu,D\varphi(x,t)\big\rangle-r\varphi(x,t).
\end{align*}

Next, we assume $D\varphi(x,t)=0$. Now, since there is no more gradient dependence in $\Phi$, the term inside $\inf \sup$ in $\Phi$ is always bounded, and hence for any $\nu \in \mathbb S^{n-1}$, there is $b_m \in \mathbb S^{n-1}$ such that
\begin{align*}
0&\geq \partial_t \varphi(x,t)+\tr\Big(\mathcal A_{\nu,b_m}^{(x,t)}D^2\varphi(x,t)\Big)+\big\langle \mu,D\varphi(x,t)\big\rangle-r\varphi(x,t) \\
&\geq \partial_t \varphi(x,t)+P^-(D^2\varphi(x,t))+\big\langle \mu,D\varphi(x,t)\big\rangle-r\varphi(x,t).
\end{align*}

Let $\phi \in C^2\big(\mathbb R^n\times (0,T)\big)$ test $u_m$ from above at $(x,t)$. In a similar way to the above, if $D\phi(x,t)\not=0$, we can find $a_m\in \mathbb S^{n-1}$ such that
\begin{align*}
0&\leq \partial_t \phi(x,t)+\tr\Big(\mathcal A_{a_m,-\frac{D\phi(x,t)}{|D\phi(x,t)|}}^{(x,t)}D^2\phi(x,t)\Big)+\big\langle \mu,D\phi(x,t)\big\rangle -r\phi(x,t) \\
&\leq \partial_t \phi(x,t)+P^+\big(D^2\phi(x,t)\big)+\big\langle \mu,D\phi(x,t)\big\rangle-r\phi(x,t),
\end{align*}
because $u_m$ is a subsolution to \eqref{eq: bdd 2}. Furthermore, if $D\phi(x,t)=0$, for any $\nu \in \mathbb S^{n-1}$, there is $a_m \in \mathbb S^{n-1}$ such that
\begin{align*}
0&\leq \partial_t \phi(x,t)+\tr\Big(\mathcal A_{a_m,\nu}^{(x,t)}D^2\phi(x,t)\Big)+\big\langle \mu,D\phi(x,t)\big\rangle-r\phi(x,t) \\
&\leq \partial_t \phi(x,t)+P^+(D^2\phi(x,t))+\big\langle \mu,D\phi(x,t)\big\rangle-r\phi(x,t).
\end{align*}
Thus, we have shown that $u_m$ is a super- and a subsolution to the equations
\begin{align*}
\begin{cases}
\partial_t u_m(x,t)+P^-\big(D^2u_m(x,t)\big)+\big\langle \mu,Du_m(x,t)\big\rangle-ru_m(x,t)=0, \\
\partial_t u_m(x,t)+P^+\big(D^2u_m(x,t)\big)+\big\langle \mu,Du_m(x,t)\big\rangle-ru_m(x,t)=0, 
\end{cases}
\end{align*}
respectively. Therefore, the classical result of \cite[Theorem 4.19]{wang92}, see also \cite{krylovs80}, implies that the function $u_m$ is H\"{o}lder continuous with a H\"{o}lder constant independent of $m$. 
\end{proof}
\end{lemma}

We are now in a position to prove the main theorem of the paper.
\begin{proof}[Proof of Theorem \ref{main theorem of the paper}]By the comparison principle Lemma \ref{uniqueness} and \eqref{g assumptions}, we see that the sequence $(u_m)$ of solutions to \eqref{eq: bdd 2} is uniformly bounded in $m$. Hence, because Lemma \ref{lemma: equicontinuity} holds, by the Arzel\`{a}-Ascoli theorem, there exist $u$, continuous on $\mathbb R^n \times [0,T]$, and a subsequence $(m_j)$ such that it holds
$$
u_{m_j} \to u
$$
uniformly on $\mathbb R^n\times[0,T]$ as $j\to \infty$. By Lemmas \ref{test f reduction} and \ref{lemma: limit eq}, the stability principle for viscosity solutions yield that $u$ is a viscosity solution to \eqref{eq: bdd 3}. Therefore by Lemma \ref{lemma: game value}, the final part is to show that the value function $U_m^-$ with uniformly bounded controls converges to the value function $U^-$ as $m \to \infty$. This follows from the properties of the infimum and the supremum, because the boundary values $g$ are bounded, and for the set of admissible strategies, it holds $\mathcal S=\bigcup_{m}\mathcal S_m$. For more details, see for example \cite[the proof of Theorem 1.2]{nystromp17}.  

The corresponding proofs in the context of $U^+$, $U_m^+$ and the equation \eqref{eq: bdd 1} are analogous to the above. In particular, let $u_m^+$ be the unique viscosity solution to \eqref{eq: bdd 1}. The proof of Lemma \ref{use ishii app} for $u_m^+$ and $F_m^+$ is essentially the same as before. Then by minor adjustments to the proofs of Lemmas \ref{lemma: game value} and \ref{main lemma add reg}, we can show that $u_m^+=U_m^+$ on $\mathbb R^n \times [0,T]$. Finally, the uniform boundedness and the equicontinuity of the family $(u_m^+)$, together with the convergence of $U_m^+$ to $U^+$ as $m\to \infty$, follows as before. Therefore, the proof is complete.

\end{proof}

\def\cprime{$'$} \def\cprime{$'$}

\end{document}